\documentclass[reqno,english,11pt]{amsart}

\usepackage{amsmath,amsfonts,amssymb,graphicx,amsthm,enumerate,url}
\usepackage[noadjust]{cite}
\usepackage{stmaryrd}
\usepackage{mathrsfs,booktabs,tabularx}
\usepackage{xifthen,xcolor,tikz,setspace}
\usetikzlibrary{decorations.pathmorphing,patterns,shapes,calc,decorations}
\usetikzlibrary{decorations.pathreplacing}
\usepackage{mathtools,upgreek}
\usepackage[final]{showkeys} 

\definecolor{refkey}{gray}{.75}
\definecolor{labelkey}{gray}{.5}
\usepackage[letterpaper]{geometry}
\usepackage[colorinlistoftodos]{todonotes}
\presetkeys{todonotes}{inline, color=green}{}

\usepackage{accents}
\usepackage[algo2e,boxed,vlined,algoruled]{algorithm2e}

\usepackage{float}
\usepackage{subcaption}

\usepackage[colorlinks=true]{hyperref}
\colorlet{DarkGreen}{green!50!black}
\colorlet{DarkGray}{gray!60!black}

\numberwithin{equation}{section}

\renewcommand{\epsilon}{\varepsilon}
\newcommand{\given}{\;\big|\;}
\newcommand{\one}{\mathbf{1}}

\usepackage{color}
 \definecolor{refkey}{gray}{.5}
 \definecolor{labelkey}{gray}{.5}
\definecolor{light}{gray}{.9}

\newtheorem{theorem}{Theorem}[section]
\newtheorem*{theorem*}{Theorem}
\newtheorem{lemma}[theorem]{Lemma}

\newtheorem{claim}[theorem]{Claim}
\newtheorem{proposition}[theorem]{Proposition}

\newtheorem{fact}[theorem]{Fact}
\newtheorem{corollary}[theorem]{Corollary}

\theoremstyle{definition}{

\newtheorem{definition}[theorem]{Definition}

\newtheorem*{definition*}{Definition}

\newtheorem{remark}[theorem]{Remark}
\newtheorem*{remark*}{Remark}

}

\newcommand{\E}{\mathbb E}

\renewcommand{\P}{\mathbb P}

\newcommand{\cA}{\ensuremath{\mathcal A}}

\newcommand{\cC}{\ensuremath{\mathcal C}}
\newcommand{\cD}{\ensuremath{\mathcal D}}
\newcommand{\cE}{\ensuremath{\mathcal E}}
\newcommand{\cF}{\ensuremath{\mathcal F}}

\newcommand{\cN}{\ensuremath{\mathcal N}}

\newcommand{\cP}{\ensuremath{\mathcal P}}
\newcommand{\cQ}{\ensuremath{\mathcal Q}}
\newcommand{\cR}{\ensuremath{\mathcal R}}

\newcommand{\cT}{\ensuremath{\mathcal T}}

\newcommand{\llb }{\llbracket}
\newcommand{\rrb }{\rrbracket}

\newcommand{\bh}{{\ensuremath{\vec{h}}}}
\newcommand{\bg}{{\ensuremath{\vec{g}}}}
\newcommand{\bm}{{\ensuremath{\vec{m}}}}
\newcommand{\bn}{{\ensuremath{\vec{n}}}}

\newcommand{\bS}{{\ensuremath{S}}}
\newcommand{\bG}{{\ensuremath{G}}}
\newcommand{\bH}{{\ensuremath{H}}}
\newcommand{\bA}{{\ensuremath{A}}}
\newcommand{\bSigma}{{\ensuremath{\Sigma}}}

 \renewcommand{\epsilon}{\varepsilon}
\DeclareMathOperator{\var}{Var}

\DeclareMathOperator{\cov}{Cov}
\DeclareMathOperator{\ber}{Ber}

\DeclareMathOperator{\Harm}{Harm}
\DeclareMathOperator{\erf}{erf}

\newcommand{\tv}{{\textsc{tv}}}

\makeatletter
\newcommand{\superimpose}[2]{%
  {\ooalign{$#1\@firstoftwo#2$\cr\hfil$#1\@secondoftwo#2$\hfil\cr}}}
  
\newcommand{\sbullet}{%
  \hbox{\fontfamily{lmr}\fontsize{.4\dimexpr(\f@size pt)}{0}\selectfont\textbullet}}

\makeatother

\begin{document}

\title{Cutoff for the Glauber dynamics of the lattice free field}

\author{Shirshendu Ganguly}
\address{S.\ Ganguly\hfill\break
Department of Statistics \\ UC Berkeley }
\email{sganguly@berkeley.edu}
\author{Reza Gheissari}
\address{R.\ Gheissari\hfill\break
Departments of Statistics and EECS \\ UC Berkeley }
\email{gheissari@berkeley.edu}

\begin{abstract}
    The Gaussian Free Field (GFF) is a canonical random surface  in probability theory generalizing Brownian motion to higher dimensions. In two dimensions, it is critical in several senses, and is expected to be the universal scaling limit of a host of random surface models in statistical physics.
    It also arises naturally as the stationary solution to the stochastic heat equation with additive noise. Focusing on the dynamical aspects of the corresponding universality class, we study the mixing time, i.e., the rate of convergence to stationarity, for the canonical prelimiting object, namely the discrete Gaussian free field (DGFF), evolving along the (heat-bath) Glauber dynamics. 
    While there have been significant breakthroughs made in the study of cutoff for Glauber dynamics of random curves, analogous sharp mixing bounds for random surface evolutions have remained elusive.
    In this direction, we establish that on a box of side-length $n$ in $\mathbb Z^2$, when started out of equilibrium, the Glauber dynamics for the DGFF exhibit cutoff at time $\frac{2}{\pi^2}n^2 \log n$.  
\end{abstract}

\maketitle

\vspace{-.75cm}
{
  \hypersetup{linkcolor=black}
\setcounter{tocdepth}{1}
\tableofcontents
}

\vspace{-1cm}
\section{Introduction}
Stochastic interfaces and naturally associated dynamics are mathematical models for various natural phenomena. Classical examples include ballistic deposition, crystal growth, and the evolution of boundaries separating thermodynamic phases. Often, the fluctuation theory of the interface is expected to be governed by one of a few canonical stochastic partial differential equations (SPDEs). Key details of the models determine the relevant SPDE, and in turn the relevant universality class, and in particular the Gaussian  or non-Gaussian nature of the fluctuations.  

This paper is concerned with the dynamical evolution of the Gaussian Free Field (GFF), perhaps the most canonical model of a random surface. Just as Brownian motion (simply the GFF in 1D) is the universal scaling limit of diffusive random curves, the GFF in dimension $d = 2$ is expected to arise as the universal scaling limit of a large class of random surfaces in lattice statistical mechanics.
Before discussing examples of the latter and their associated dynamics that have been the object of intense investigation over the years, let us mention that there are various other perspectives to which the GFF is central including in conformal field theory, and in Coulomb gas theory and the study of log-correlated fields: see the surveys~\cite{Sheffield-GFF-notes,Biskup-notes,Berestycki-notes} for more on these angles. 

A canonical dynamical evolution of the GFF is best described by a \emph{height function} $h(x,t):\Lambda \times \mathbb R_+ \to \mathbb R$ for some domain $\Lambda \subset \mathbb R^d$, that is evolving stochastically with a local smoothening force; this is encoded by the stochastic heat equation with additive noise (SHE),
\begin{equation}\label{eq:she}
\partial_t h(x,t)=\Delta_x h(x,t)+\eta(x,t)\,,
\end{equation}
where the (spatial) Laplacian $\Delta_x$ is the local smoothening operator and $\eta$ is a space-time white noise. The stochastic heat equation is a widely studied SPDE, and it is a classical fact that the GFF is its stationary solution (see e.g. \cite{hairer-SPDE-notes}).

In two spatial dimensions, the SHE and its stationary solution, the GFF, form the central member of what is sometimes referred to as the Edwards--Wilkinson (EW) universality class of surface growth, named after a physical model for granular surface aggregation introduced in~\cite{Edwards-Wilkinson}.
The EW universality class is expected to encompass both the equilibrium and off-equilibrium fluctuations of a host of random surface models arising in statistical physics; 
by way of example, this class is expected to include the 
Ginzburg--Landau $\nabla \varphi$ (GL) interface model, the height functions of the dimer and six-vertex models, the solid-on-solid model at high temperatures, and the interface between phases in the 3D Ising model in its roughening regime.
In the case of the dimer model and the GL model the equilibrium fluctuations are known to converge to the GFF (e.g.,~\cite{Kenyon-dominoes-free-field} and~\cite{Miller-Ginzburg-landau} respectively), and in some of the other cases, it is at least known that the variance of the fluctuations is of the same order as that of the GFF~\cite{Frohlich-Spencer,log-variance-square-ice,DC-delocalization-six-vertex}.

In light of the above, it is of much interest to study dynamical aspects of this universality class. We will focus on the central lattice model identifying this class: the 2D \emph{discrete Gaussian free field} (DGFF). This is a discretization in space of the GFF, where $h(x,t)$ lives on $\Lambda_n\times \mathbb R_+$, where $\Lambda_n$ is a box of side-length $n$ in the integer lattice $\mathbb Z^2$. 
While there are various examples of natural discrete dynamic evolutions mimicking~\eqref{eq:she} with the Laplacian replaced  by the discrete Laplacian, and $\eta$ living on $\Lambda_n \times \mathbb R_+$, we will consider the heat-bath Glauber dynamics, also known as the Gibbs sampler, for the DGFF. This is the canonical local update Markov process where at any time, the height at a site of $\Lambda_n$ is updated according to~\eqref{eq:she}, i.e., taking value given by the average of its neighboring values plus an independent standard Gaussian. 

Of particular interest is the analysis of this dynamics when the surface is initialized far from equilibrium, say at an atypically high height of $n$. Under a macroscopic rescaling, the random surface evolution becomes non-random admitting a hydrodynamic limit to a non-linear parabolic PDE:
such results have been shown for certain dynamics of the form of~\eqref{eq:she} in~\cite{Funaki-Spohn,Nishikawa} for the GL model (of which the DGFF is a special case),  in~\cite{Laslier-Toninelli,Toninelli-ICM-review} for the height function of the dimer model, and in~\cite{LST-zero-temperature-2d-ising} for the zero-temperature Ising interface in $\mathbb Z^2$.
However, one cannot hope to extract information about convergence of the random process to equilibrium from the deterministic large-scale behavior alone. This finer resolution is the focus of this paper. 

\begin{figure}[t]
\begin{subfigure}[b]{.31\textwidth}
\begin{tikzpicture}[scale = 1]
\node at (0,0) {\includegraphics[width = 1.55in]{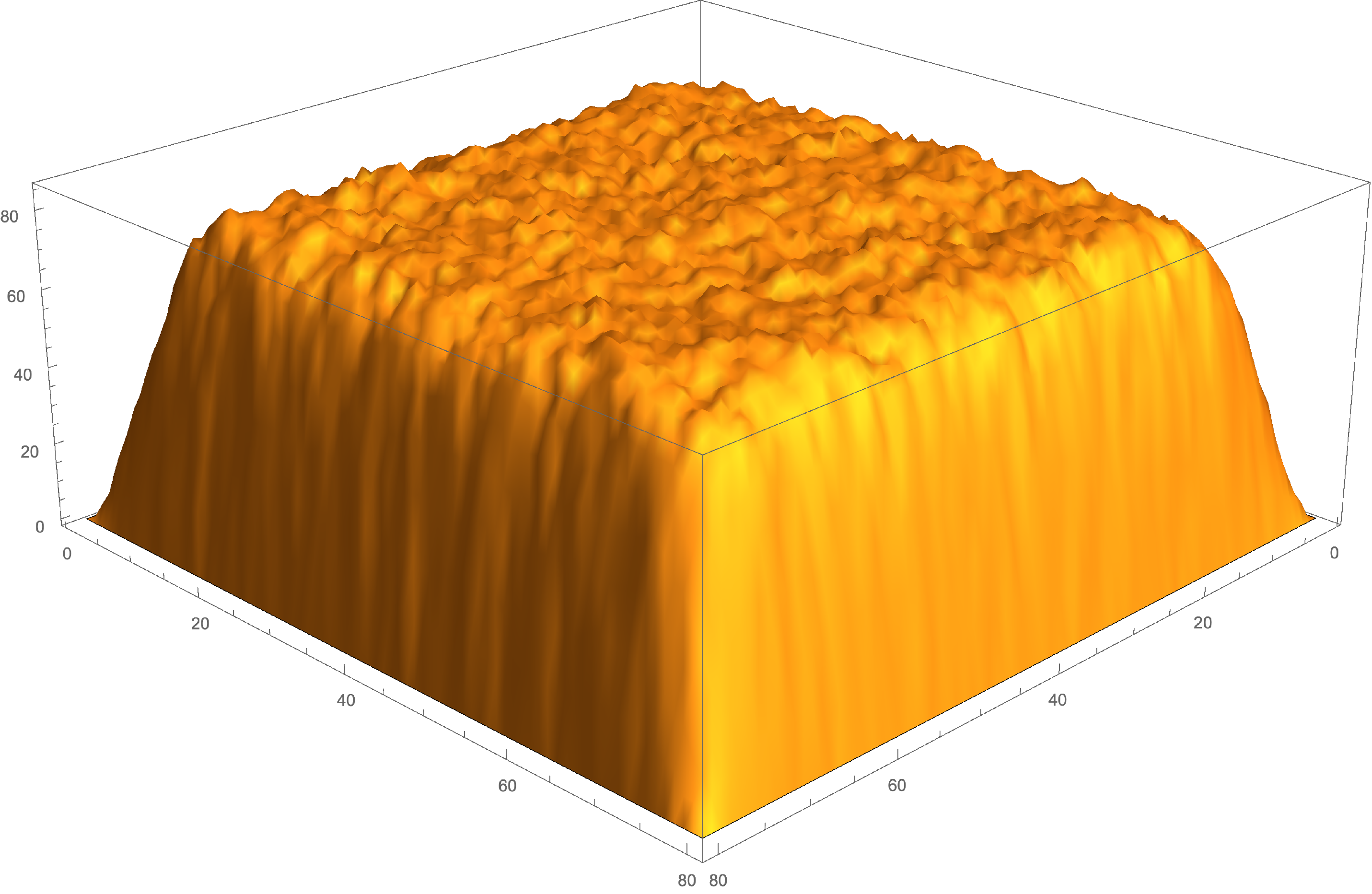}};
\node at (0,-3.5) {\includegraphics[width = 1.55in]{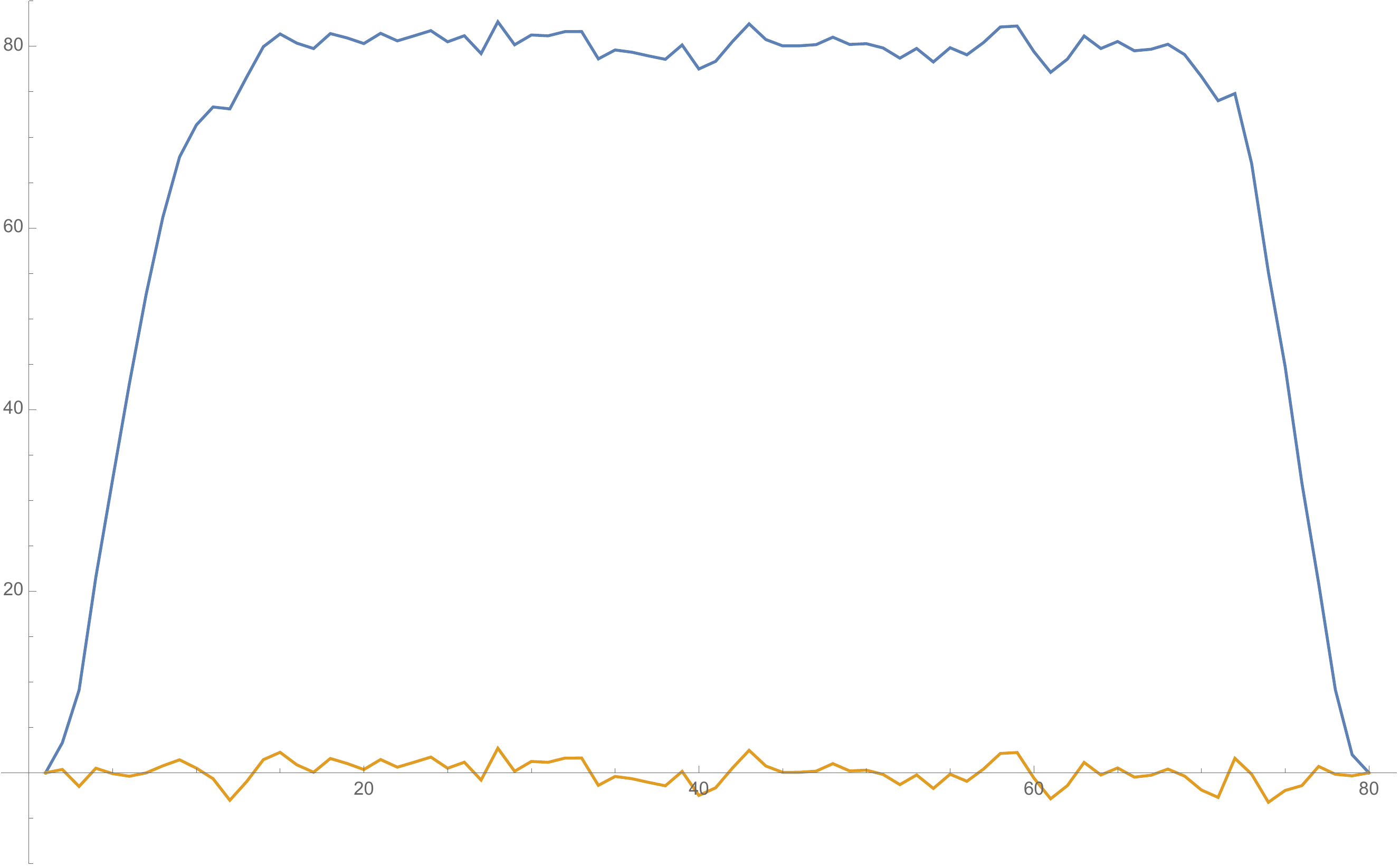}};
\end{tikzpicture}
\subcaption{$t=20$}
\end{subfigure}
\begin{subfigure}[b]{.31\textwidth}
\begin{tikzpicture}[scale = 1]
\node at (0,0) {\includegraphics[width = 1.55in]{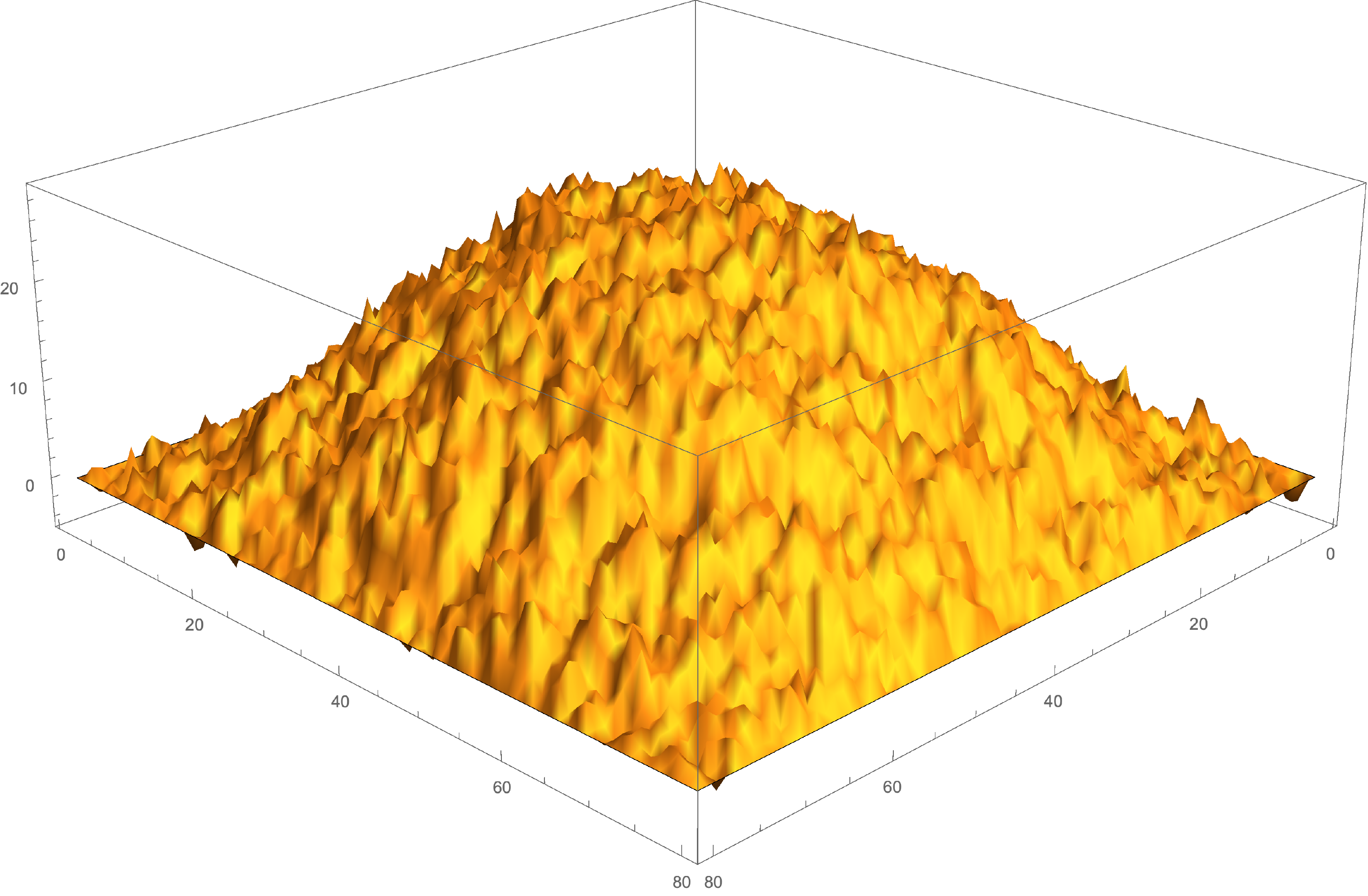}};
\node at (0,-3.5) {\includegraphics[width = 1.55in]{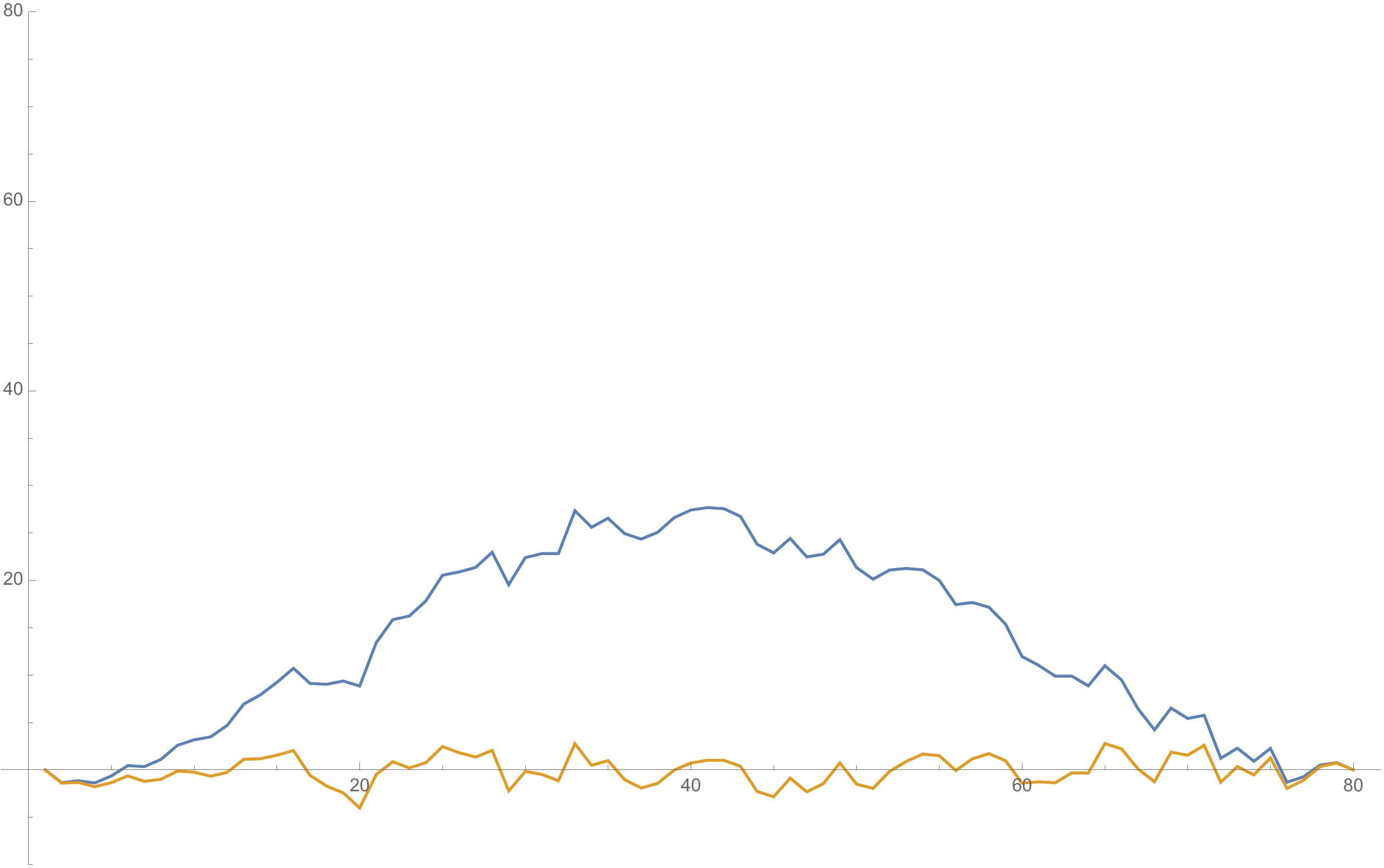}};
\end{tikzpicture}
\subcaption{$t=1000$}
\end{subfigure}
\begin{subfigure}[b]{.31\textwidth}
\begin{tikzpicture}[scale = 1]
\node at (0,0) {\includegraphics[width = 1.55in]{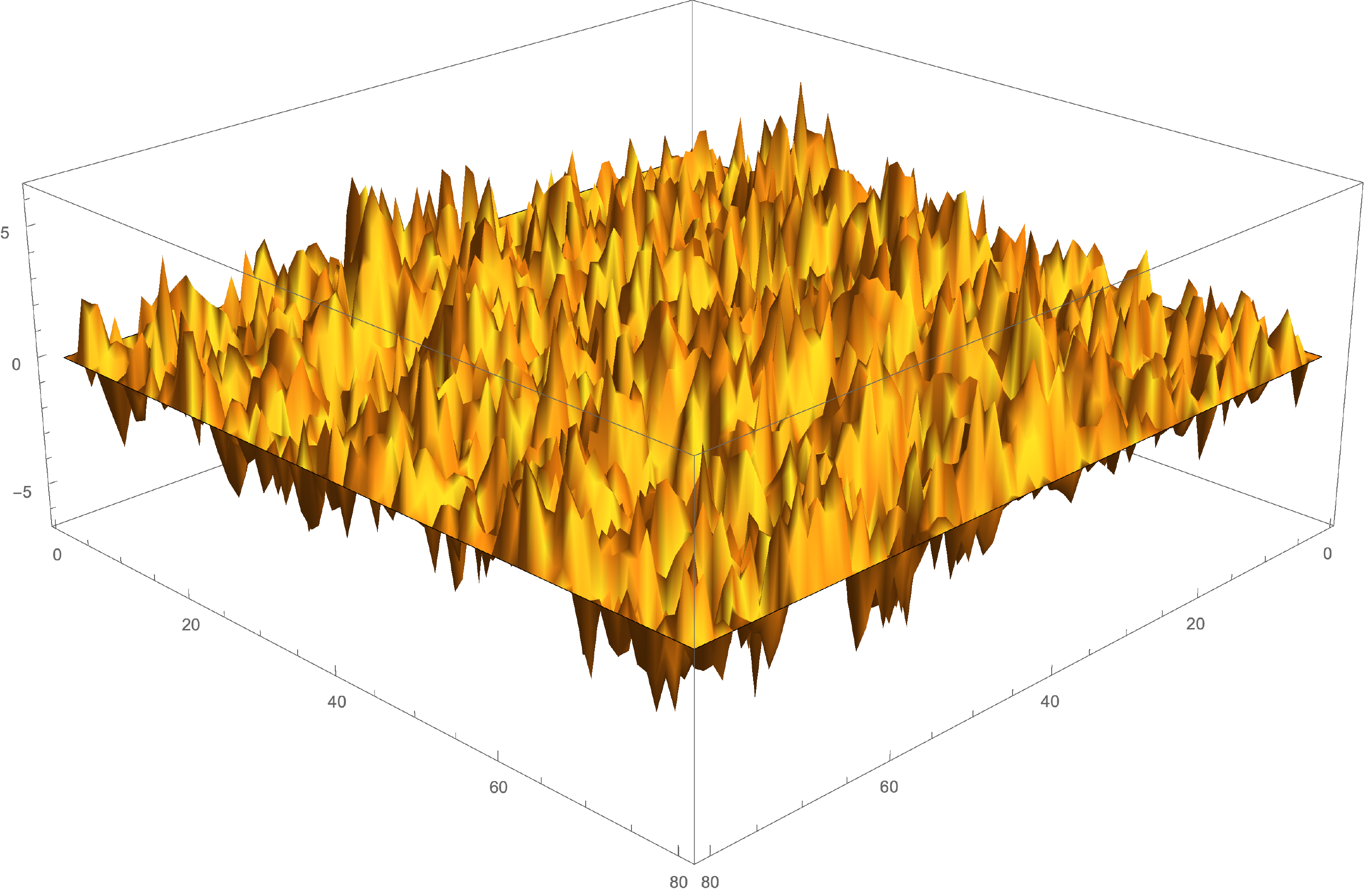}};
\node at (0,-3.5) {\includegraphics[width = 1.55in]{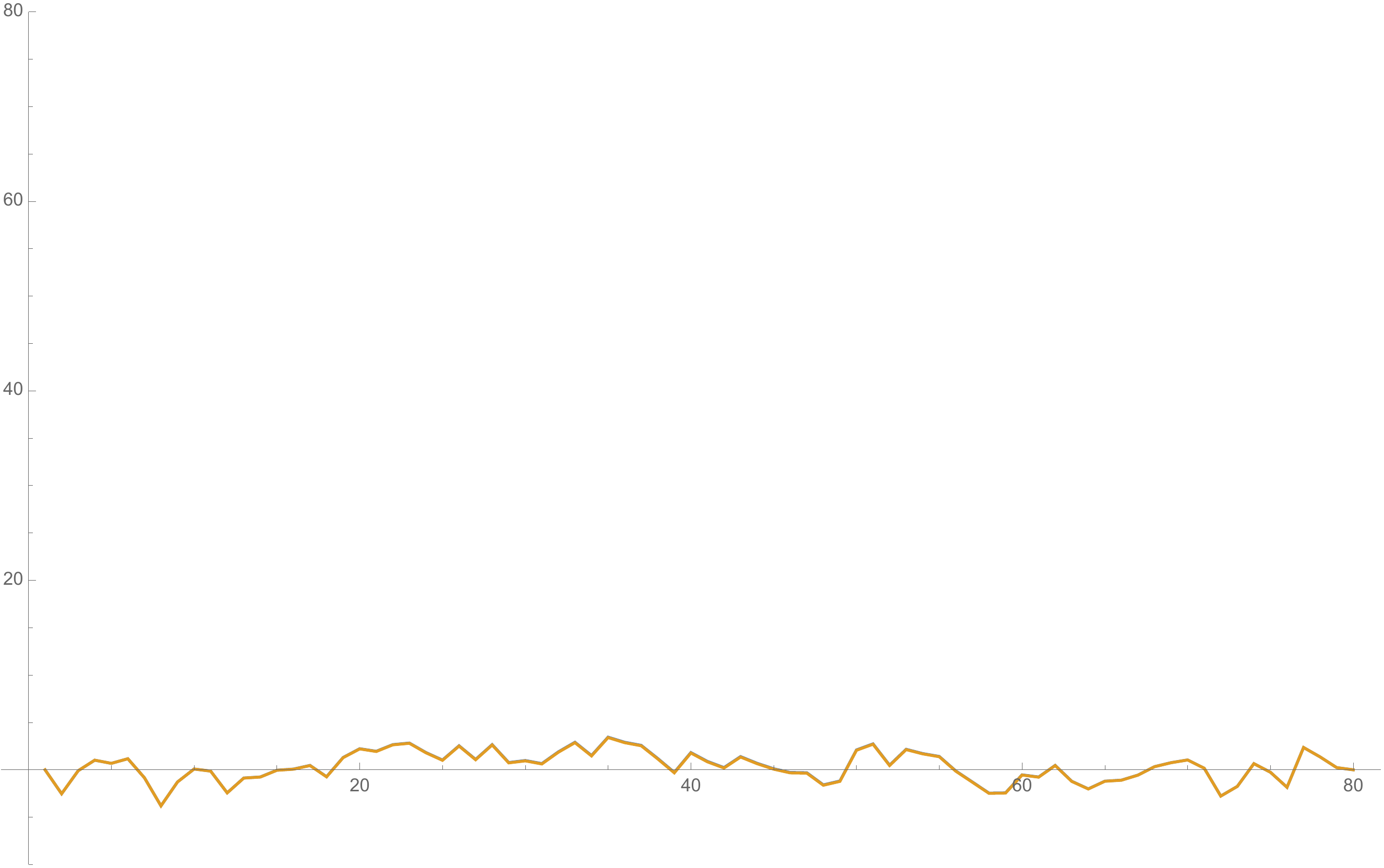}};
\end{tikzpicture}
\subcaption{$t=5000$}
\end{subfigure}
    \caption{Top: three snapshots of the DGFF Glauber dynamics on an $n\times n$ box with $n=80$, initialized from the constant surface at height $n$; after (continuous) time $1000$, it globally resembles a multiple of the top eigenvector of the discrete Laplacian, and evolves according to the mean-curvature flow. By time $5000$, its mean is zero, and its fluctuations appear equilibrated. Bottom: Cross-sections of the above surfaces along the $x=y$ diagonal (blue), together with the same cross-sections of a stationary DGFF evolution (orange). When $t=5000$, the two cross-sections are fully coupled, in agreement with the claim that the surface has equilibrated.}
    \label{fig:DGFF-Glauber-evolution}
\end{figure}

Setting up the right framework to study such questions naturally brings us into the world of mixing times of Markov processes. This field has been the subject of intense investigation over the last several decades, and is too large to do justice to in this introduction, so we simply mention that much attention has been paid to mixing times of Glauber dynamics for lattice statistical mechanics models, including some of the aforementioned models belonging to the EW universality class. 

Often in high-dimensional Markov chains, a remarkable phenomenon, known as \emph{cutoff}, occurs where the total-variation distance of the system to equilibrium dramatically drops from $1$ to $0$ in a window of time that is smaller order than the mixing time itself. First identified in seminal works for the random walk on the hypercube and the symmetric group (see~\cite{Aldous,Aldous-Diaconis,Diaconis-PNAS}), establishing cutoff and its location has become the ``best possible" result in the analysis of fast mixing Markov chains. We refer to the texts~\cite{Aldous-Fill,LP} for a detailed account of the literature on the cutoff phenomenon.

In our context,~\cite{Wilson} developed a general framework for showing that the mixing time for Glauber dynamics of random curves of length $n$ (i.e., $d=1$) following~\eqref{eq:she} is of the order of $n^2 \log n$. Subsequently,~\cite{Lacoin-adjacent-transposition} established cutoff at $\frac{1}{\pi^2} n^2\log n$ for the key example of the height function associated to the simple symmetric exclusion process on a line segment of length $n$. This has since been extended to a general theory of cutoff for dynamics of random curves following~\eqref{eq:she} in~\cite{CLL-nabla-phi}. When $d\ge 2$, however, the approaches followed in these papers face significant barriers, and sharp mixing time results for random surface evolutions are very limited. 

This state of affairs motivates the present work, 
where we consider the 2D DGFF and establish cutoff for its Glauber dynamics. To the best of our knowledge, this is the first cutoff result for a random surface Glauber dynamics in any dimension $d> 1$. Postponing further discussion, let us now formalize our main result. 

\subsection{Main result}
Let $\Lambda_n = \llb 1,n-1\rrb^2 = \{1,...,n-1\}^2$ be the $2$-dimensional lattice cube of side-length $n$ with nearest neighbor edges, and let $\partial \Lambda_n = \{v\in \mathbb Z^2 \setminus \Lambda_n: d(v,\Lambda_n) = 1\}$ be its (outer) boundary. Let $\overline \Lambda_n = \Lambda_n \cup \partial \Lambda_n = \llb 0,n\rrb^2$ and write $x\sim y$ if $d(x,y) =1$. 

For a fixed boundary condition $\eta: {\partial \Lambda_n} \to \mathbb R$, the DGFF on $\Lambda_n$ is the Gaussian process $\bh=(h(x))_{x\in \overline \Lambda_n}$ with $h(x)=\eta(x)$ for all $x\in \partial \Lambda_n,$ whose density $\pi_n$ is proportional to 
\[\exp\Big(-{\frac 1{8}}\sum_{x,y\in \overline \Lambda_n: x\sim y} (h(x)-h(y))^2\Big)\,.\]
For convenience take $\eta$ to be identically zero (to see that this does not lead to any loss of generality cf.\ Remark~\ref{rem:bc}), and view any $\bh$ taking value zero on $\partial \Lambda_n$ simply as a function on~$\Lambda_n$. 

The continuous-time Glauber dynamics for the DGFF, denoted $\bh(t) = (h(x,t))_{x\in \Lambda_n}$ is defined as follows. Assign each $x\in \Lambda_n$ a rate-1 Poisson clock; if the clock at $x$ rings at time $t$, leave $h$ unchanged everywhere except at $x$ where it updates to a unit variance Gaussian whose mean is equal to the average of $h(y,t)$ for $y\in \overline \Lambda_n: y\sim x$ (see Definition~\ref{def:surface-RW-coupling} for a formal definition). 

Next, we formally introduce the notions of mixing time and cutoff. 
Recall the total variation distance $\|\mu - \nu\|_\tv$ between probability measures $\mu$ and $\nu$ defined on the same space. Let $P_t \delta_{\bh(0)}$ be the law of $\bh(t)$, when initialized from some $\bh(0)$ taking value zero on $\partial \Lambda_n$. Note that in order for the mixing time (the time it takes for the total variation distance to $\pi_n$ to be less than $1/4$) to be finite, we need to restrict our attention to initializations of finite height; in analogy with interfaces from dimer and Ising models, a natural such choice is $\bh(0)$ such that ${\|\bh(0)\|}_\infty \le n$.

Define the (worst case) total-variation distance of the DGFF Glauber dynamics at time $t$ to be $$d_{\tv}^{(n)}(t):=\max_{\bh(0): \|\bh(0)\|_\infty \le n} \| P_t \delta_{\bh(0)} - \pi_n \|_\tv\,.$$

\begin{definition}\label{def:cutoff}
We say the Glauber dynamics for the DGFF on $\Lambda_n$ with initialization $\bh(0)$ such that $\|\bh(0)\|_\infty \le n$ undergoes cutoff at time $t_n$ with window $s_n = o(t_n)$ if 
\begin{align*}
    \lim_{\lambda\to\infty} \liminf_{n\to\infty}  d_{\tv}^{(n)}(t_n - \lambda s_n) = 1\,, \quad \mbox{and}\quad \lim_{\lambda \to\infty} \limsup_{n\to\infty} d_\tv^{(n)} (t_n + \lambda s_n) = 0\,.
\end{align*}
\end{definition}

With the above preparation, we now state our main result.

\begin{theorem}\label{thm:main-revised}
The Glauber dynamics for the DGFF on $\Lambda_n$ with initializations $\bh(0): \|\bh(0)\|_\infty \le n$ and zero boundary conditions, exhibits cutoff with window $O(n^2 \log \log n)$ at
\begin{align}\label{eq:t-star}
    t_\star &  :=
    \frac{2}{\pi^2} n^2 \log n\,.
\end{align}
\end{theorem}

\begin{remark}\label{rem:general-initializations}
    We chose the total-variation to be maximized over $\|\bh(0)\|_\infty \le n$ for concreteness, but if we chose to instead maximize over $\|\bh(0)\|_\infty \le a_n$ for any sequence $a_n\ge 0$, then we would find that the mixing time maximized over such initializations is $t_\star(a_n) + O(n^2 \log \log n)$, where  
    \begin{align}\label{eq:t-star-an}
        t_\star(a_n) := \frac{2}{\pi^2} n^2 \log a_n\,.
    \end{align}
    This implies the cutoff phenomenon if $\frac{\log a_n}{\log \log n} \to \infty$.  
    
    If one focuses on the special case of the flat initialization $\bh(0)\equiv 0$, then we deduce that its mixing time is in fact $O(n^2 \log \log n)$; this specific initialization would be expected to mix in time $O(n^2)$ with no cutoff, but even establishing a quantitative $o(n^2 \log n)$ bound on its mixing time is challenging in $d=1$.   
\end{remark}

For the total-variation lower bounds in the above results, we can go further and obtain sharp lower bounds for the cutoff profile by leveraging the Gaussian nature of the process. This cutoff profile is expected to be sharp, but we are unable to get a matching upper bound due to serious technical obstructions: see Section~\ref{subsec:proof-ideas} and Remark~\ref{rem:tv-Gaussians} in particular for more.  

In the below, let $\erf(x)$ be the error function, i.e., the probability that a mean-zero Gaussian with variance $1/2$ falls in the range $[-x,x]$. 

\begin{theorem}\label{thm:main-lower-bound}
    The Glauber dynamics for the DGFF on $\Lambda_n$ with initializations $\bh(0)$ having $\|\bh(0)\|_\infty\le a_n$ satisfying $\frac{a_n}{\log n} \to\infty$ and zero boundary conditions satisfies 
    \begin{align*}
        \liminf_{n\to\infty} d_\tv^{(n)}(t_\star(a_n) + sn^2) \ge \erf\Big( \frac{2}{\pi} e^{ - \pi^2 s/2}\Big)\,.
    \end{align*}
\end{theorem}

\begin{remark}\label{rem:bc}
Though we focused on the case of zero boundary data, our results hold for general Dirichlet boundary data $\eta$. In fact the measure $\pi_n$ as well as the Glauber dynamics with boundary condition $\eta$ is obtained from the one with zero boundary conditions by a deterministic $\eta$-dependent translation. 
See Lemma~\ref{lem:boundary-conditions} for the formal statement. 

Moreover, though our results were formulated for the continuous-time Glauber dynamics, it will be clear in the proofs that we could just as well have written them for the discrete-time Glauber dynamics, and the cutoff location would become $n^2 t_\star$ with $O(n^4 \log \log n)$ window. 
\end{remark}

\subsection{Related work}\label{subsec:related-work}
Before describing the key ideas in the paper,  we briefly recount related cutoff results in one dimension, and the more limited results on random surface dynamics in higher dimensions. The most basic 1D interface model is a $\pm 1$ random walk path pinned at height zero at points $0$ and $n$. In the associated Glauber dynamics, a local minimum $\vee$ (resp. local maximum~$\wedge$) at a point $x\in \llb 1,n-1\rrb$  flips with probability $1/2$ at the corresponding clock rings.

This seemingly innocuous dynamics encodes many apparently more complicated Markov chains including the adjacent transposition walk on the symmetric group, and the symmetric exclusion process on $\llb 1,n-1\rrb$. The famous work of Wilson~\cite{Wilson} gave upper and lower bounds of matching order $n^2 \log n$ for the mixing time of these processes; a key contribution of that work was noticing that the expected value of the height at a point evolves according to the discrete heat equation, whose spectral theory on the line segment is of course well-understood. 

The bounds of~\cite{Wilson} remained the state of the art until the breakthrough work of Lacoin~\cite{Lacoin-adjacent-transposition} establishing cutoff for the above-described dynamics. That proof relied on a delicate combination of tools from the analysis of monotone Markov chains, together with martingale techniques. 
Related ideas and refinements have been subsequently developed in e.g.,~\cite{Lacoin-diffusive-window,CLL-mixing-simplex} leading to a robust theory of cutoff for a large class of 1D interface models in the universality class of~\eqref{eq:she}. In its most general form~\cite{CLL-nabla-phi}, cutoff at $\frac 1{\pi^2} n^2 \log n$ was established for Glauber dynamics of 1D interface models with Gibbs distribution $\exp({ -V(|\nabla h|)})$ for convex $V$. The DGFF in dimension one fits directly into this framework via the choice $V(x) = x^2/4$.

Moving on to higher dimensions, arguments similar to those of~\cite{Wilson} can be used to see that the mixing time of the DGFF Glauber dynamics is of order $n^2 \log n$. A straightforward generalization of the argument in~\cite{CLL-nabla-phi} (which itself extended the discrete argument from~\cite{Wilson} to a continuum setup) would allow one to identify the spectral gap of the DGFF dynamics as $1-\cos(\pi/n)$. However, a key step in all proofs of cutoff from~\cite{Lacoin-adjacent-transposition,CLL-mixing-simplex,CLL-nabla-phi} breaks down as soon as $d\ne 1$. Namely, all those proofs relied on fast mixing on an $O(1)$-sized skeleton to split the segment into a finite number of smaller scale segments between the points of the skeleton, which mix independently. Martingale techniques based on connections to the discrete heat equation can be combined with monotonicity arguments to obtain useful \emph{one-point} mixing time bounds, sufficient for the mixing time of the skeleton provided the skeleton size is $O(1)$. However this will only be true in dimension one, and any such approach thus runs into a serious barrier when $d\ge 2$.

For random surfaces whose heights take integer values, the situation in $d\ge 2$ is even worse, as the expected height no longer follows the discrete heat equation due to integer effects. 
A general methodology was used in~\cite{CPT-monotone-surfaces} to prove $n^2 (\log n)^C$ bounds on mixing times of random surfaces by trying to mimic the mean-curvature flow evolution of the surface, but this required a key local input bounding the mixing time from initializations whose heights are on the same scale as the fluctuations. Such an input is only available in a few settings, notably the zero-temperature Ising interface, and the height function of the dimer model~\cite{CMST-zero-temperature-3d-ising}. We also mention recent progress by~\cite{BaBo} establishing the order $n^{-2}$ spectral gap for a hierarchical approximation to the discrete Gaussian model (the DGFF constrained to take integer valued heights). For many other important models of random surfaces like the solid-on-solid model, and the low-temperature Ising interface in three dimensions, there do not even exist sub-exponential bounds on the mixing time. 

\subsection{Idea of the proof}\label{subsec:proof-ideas}
Here, we describe our approach to proving Theorem~\ref{thm:main-revised}. 

\subsubsection{Backwards-in-time random walk representation}
The starting point of our proof is the observation that the Glauber dynamics for the DGFF admits an exact representation in terms of backwards-in-time random walks on $\Lambda_n$ killed at $\partial \Lambda_n$. This can be viewed as a discrete version of the Feynman--Kac representation for the SHE. A similar graphical representation was used to prove ergodicity and bound covariance growth in infinite volume in~\cite{FerrariNiederhauser}, and a synchronous version of this appeared in a recent work of Chatterjee~\cite{ChatterjeeRandom} for more general random surface evolutions. 

Formally, the random walk representation goes as follows. First consider the following standard graphical construction of the Glauber dynamics $(h(x,t))_{x\in \Lambda_n, t\ge 0}$. Begin by sampling a space-time Poisson process $\cP$ on $\Lambda_{n}\times [0,\infty]$. Further for each point $(x,t)\in \cP$ let there be an associated standard Gaussian $Z(x,t).$  Given this, the conditional (on $\cP$) process $(h^\cP(x,t))_{x,t}$ is constructed by at time $(x,t)\in \cP$, updating the value at $x$ to $\frac{1}{4}\sum_{y\sim x} h(y,t^-) + Z(x,t)$. 

Now for any fixed $t$, we can define the random walk started at $(x,t)$ going back in time, that whenever it encounters a point in $(y,s)\in\cP$ jumps to a uniformly random neighbor of $y$, and gets frozen (killed) when it hits $\partial \Lambda_n$.
For any such random walk trajectory $(S^{x,\cP}_s)_{s\in [0,t]}$ initialized at $S^{x,\cP}_t=x$, let $Z_t(S_{[0,t]}^{x,\cP})$ be the sum of the Gaussian variables $Z(y,s)$ encountered along the trajectory. 
The random walk representation of $(h^\cP(x,t))_{x,t}$ is the following exact equality: for any $t>0$ and any realization of $\cP$,
$$h^\cP(x,t)=\E\big[Z_t(S_{[0,t]}^{x,\cP})+h(S^{x,\cP}_{0},0)\big] \qquad \mbox{for all $x\in \Lambda_n$}\,,$$ where the expectation is taken only over the random trajectory of $S^{x,\cP}$. 
In particular, the distribution of $h^\cP(\cdot,t)$ given $\cP$ is a multivariate Gaussian and hence the unconditional law of the Glauber dynamics is a mixture, over the Poisson clock ring times, of Gaussian processes. We refer to Theorem~\ref{thm:rw-representation-exact-equality} and Corollary~\ref{cor:random-walk-representation} for the formal statements. 

\begin{remark}\label{rem:tv-Gaussians}
At this point, one could hope to leverage exact formulae for the total-variation between high-dimensional Gaussian distributions. If for most realizations of $\cP$, $h^\cP(t)$ is mixed after time $t_\star$, then the desired mixing time upper bound would follow. Using the BTRW representation, it is possible to show that for typical $\cP$, for $t\ge t_\star$, the covariance matrix of $h^\cP(t)$ can be assumed to be the Green's function (up to negligible error terms). In that case, bounding the total-variation distance between $h^\cP(t)$ and stationarity can be found, by reduction to a 1D Gaussian computation, to be equivalent to bounding
\begin{align}\label{eq:m-Laplacian-m}
     (m^\cP_t)^{\mathsf T} \Delta m^\cP_t  \qquad \mbox{where} \quad (m^\cP_t)_x = \mathbb E[h^\cP(x,t)]\,.
\end{align}
However, the mean vector $m_t^\cP$---essentially the probability that the BTRW from $x$ did not hit the boundary $\partial \Lambda_n$ in time $[0,t]$---may be quite oscillatory for fixed (typical) $\cP$ (even if averaged over $\cP$ it is very smooth). This possibility is captured by the fact that the number of jumps taken by independent BTRWs from neighboring sites $x$ and $y$, may differ by $\sqrt{t}\approx n\sqrt{\log n}$. Ignoring the correlations between the number of jumps and the trajectory of the BTRW in $\cP$, this would indeed induce oscillations in $m_t^\cP$ that are too large to naively bound. Though unlikely, these oscillations could in principle align with high modes of the Laplacian, and contribute non-negligibly to the quantity in~\eqref{eq:m-Laplacian-m}. Decoupling the dependencies between $\cP$ and the survival probability of the BTRW appears to be a challenging task, though this could yield a path to obtaining a cutoff profile upper bound to match Theorem~\ref{thm:main-lower-bound}.
\end{remark}

\subsubsection{A two-stage coupling}
Given this, the upper bound of Theorem~\ref{thm:main-revised} takes up most of the work of the paper; our approach to this upper bound entails coupling the evolution started from the maximal configuration of all $n$ and any generic configuration $g$ with $\|g\|_{\infty}\le n$. 
Let us denote the corresponding Glauber dynamics by $h^{\vec{n}}(x,t)$ and $h^{\bg}(x,t)$ respectively.

The coupling will proceed by stitching together two stages of monotone couplings so that we have $h^{\vec{n}}(x,t) \ge h^{\bg}(x,t)$ for all $t\ge 0$. As a result of this, the coalescence time of the coupling time is upper bounded by the hitting time of zero by the volume process $V_t=\sum_x h^{\vec{n}}(x,t) - h^{\bg}(x,t)$.   
We now describe the different steps by which this hitting time is bounded by $t_\star + O(n^2 \log \log n)$. Such a two-stage coupling was already used to great effect in the 1D evolutions studied in~\cite{CLL-mixing-simplex,CLL-nabla-phi}; in the below we overview the approach for pedagogical reasons, then emphasize the places where significant innovations were required to push those arguments to higher dimension. 

\begin{itemize}
\item The first phase of the coupling of $h^\bn(x,t)$ and $h^\bg(x,t)$ uses the same underlying Poisson process determining the update sequence, and the same standard Gaussian random variables $Z(y,s)$ to carry out the updates at each site. This is called the identity coupling. (The specifics of the coupling beyond monotonicity are actually not important in the first stage, but the identity coupling is the simplest such choice.)

This stage is used to smoothly bring the mean of the process $V_t$ down to slightly below (e.g., by a factor of $(\log n)^{-5}$) its ``equilibrium scale" of $O(n^2)$. This occurs in time that is $t_\star + O(n^2 \log \log n)$ because the expectation of the volume process $V_t$ shrinks from its initial value of $n^{3}$ with exponential rate $\lambda_\one = \frac{\pi^2}{2n^2}+O(n^{-4})$, the smallest eigenvalue of the Laplacian associated to the standard random walk on $\Lambda_n$. 
This follows from the BTRW representation, by recognizing that the mean discrepancy at a site $x$ is connected to the  probability that the BTRW $S_{-t}^{x,\cP}$ has not been absorbed into $\partial \Lambda_n$---what we call the survival probability of the BTRW. By means of this correspondance, we get both quenched and annealed in $\cP$ estimates on the mean of the DGFF dynamics, and thus the mean discrepancy at each site. 

\item If we were to continue using the identity coupling, these discrepancies would continue decaying with rate $\lambda_\one$, and the time it would take to reach $V_t = o(1)$ so that a ``union bound" over all sites would suffice for coalescence, is a further $\frac{4}{\pi^2}n^2 \log n$---destroying any chance at establishing cutoff. We therefore switch at this point to a coupling called \emph{the sticky coupling}, which again uses the same Poisson clocks for the two chains, but upon an update at a site $x$, couples the resulting values of the update at $x$ optimally (maximizing the probability that the discrepancy at $x$ becomes zero).  

The argument then proceeds by establishing that under the sticky coupling, in a further $O(n^2)$ time, $V_t$ hits zero. For this, we leverage the fact that $V_t$ is a super martingale, then perform a multi-scale analysis, at each scale bounding the time it takes to drop by a multiplicative factor of $1/\log n$. The $i$'th such scale reduction will take time at most $2^{-i} n^2$ so that this can be summed over $i$ and cumulatively contributes a time of $O(n^2)$. 

Unsurprisingly, this analysis relies on lower bounding the angle bracket process corresponding to $V_t$ (for the reader unfamiliar with the notion, the precise definition appears in Definition~\ref{def:supermartingale-angle-bracket}), which in turn upper bounds the hitting times of lower scales. While this is a rather technical part of the proof and bears similarities to the arguments in \cite{CLL-mixing-simplex,CLL-nabla-phi}, adapting the bound in~\cite{CLL-nabla-phi} completely breaks when the dimension is increased to two; this is because of a use of Cauchy--Schwarz to lower bound the growth rate of the angle bracket---governed by the $\ell^2$-norm of the discrepancy process $h^\bn(x,t) - h^\bg(x,t)$---by the volume process $V_t$. That lower bound incurs a factor of $1/|\Lambda_n|$, which in dimension two (and higher) is too small to still bound the hitting times of smaller scales by $o(n^2\log n)$. We refer the reader to Remark~\ref{rem:quadratic-variation-bound-difficulty} for a precise description of this obstruction.  

Our workaround to this problem is to use a tight bound on the number of sites at which the discrepancy is non-zero. To do so, we rely on the observation that the sticky coupling, when it fails to couple two updates, tends to yield a discrepancy that is at least of order one. Leveraging such a stochastic domination, we establish a concentration inequality (Lemma~\ref{lem:two-norm-to-one-norm}) comparing the $\ell^2$ norm of the discrepancy process to its number of non-zero discrepancies, and then in turn comparing that to its $\ell^1$ norm, $V_t$, that holds throughout the second stage of the coupling process.  
This comparison also uses as an input an $\ell^\infty$ bound on the maximal discrepancy through the process, which is achieved by means of the BTRW representation and our knowledge of the (quenched) Gaussian law of the DGFF dynamics. Obtaining regularity estimates of this form that hold throughout the evolution, is often an esential source of difficulty in analysis of random surface dynamics.  
\end{itemize}

\begin{remark}\label{rem:intro-higher-d}
Let us remark briefly on the trouble this argument would face if considering the DGFF in dimensions $d\ge 3$. There, even under optimal assumptions on the quadratic variation accumulated by the $V_t$ super martingale under the sticky coupling, the coalescence time of the sticky coupling would be of order $\frac{2d-2}{\pi^2}n^2 \log n$, instead of the expected $\frac{d}{\pi^2} n^2 \log n$. For more details, see Remark~\ref{rem:higher-d-sticky-coupling}. Thus, any proof of cutoff for dimensions $d\ge 3$ must go either by a completely different dynamical coupling, or by more analytic means like those hinted at in Remark~\ref{rem:tv-Gaussians}.
\end{remark}

 \subsubsection{Sharp lower bound on the cutoff profile}
For the lower bounds of Theorem~\ref{thm:main-revised} and~\ref{thm:main-lower-bound}, we use as a test function the inner product of the first eigenfunction of the Laplacian with $h$, when initialized from a deterministic shift of a sample from stationarity. Conditioned on $\cP$, the law of this function is Gaussian, whose mean and covariance we can describe precisely, up to $o(1)$ error terms. This test function captures the right cutoff profile because the higher Fourier modes of $h$ should ``mix" in time that is $o(t_\star)$ when started out of equilibrium, and to first order, the mean of $h$ is simply a multiple of the first eigenfunction of the Laplacian. In~\cite{Lacoin-cutoff-circle}, the cutoff profile was obtained for the symmetric exclusion process on the circle ($d=1$) and similarly came from the inner product of the height with the first eigenfunction of the corresponding Laplacian. 

We note that the above intuition is difficult to quantify into an upper bound on the cutoff profile in our setting, however, because the $o(1)$ fluctuations in the mean of $h$ can in principle align with high modes of the Laplacian, in which case they would contribute significantly to the total-variation distance between two DGFF evolutions. 

\subsection{Related questions}\label{subsec:related-questions}
We conclude the introduction with a brief account of possible extensions and other related themes of interest. 
Given our exact usage of the inner product with the first eigenfunction of the Laplacian to get the cutoff profile lower bound, it may be of interest to obtain a limit theory for the evolution of the Fourier coefficients of the height function $\bh(t)$. This could offer a path for controlling high modes of the random surface, and obtaining an upper bound on the cutoff profile that matches Theorem~\ref{thm:main-lower-bound}. If we were studying a Langevin diffusion process emulating~\eqref{eq:she}, then the problem would diagonalize, and the Fourier coefficients would truly evolve according to \emph{independent} Ornstein--Uhlenbeck (OU) processes whose mean-reversion factors depend on the index of their Fourier mode. The tensorization would simplify the study of its mixing time and cutoff greatly (c.f.\ the recent paper~\cite{BCL22} where cutoff for \emph{interacting} OU processes was studied in detail). Since our motivation comes from the study of random surface dynamics, where the allowed heights are often discrete, our preference is for the Glauber dynamics, where this tensorization does not hold. But it would be interesting to understand to what extent correlations induced by $\cP$ in the Fourier coefficient evolutions disappear in the limit.

Beyond the Glauber dynamics of the DGFF, of course even obtaining $O(n^2 \log n)$ mixing time bounds, let alone cutoff, for lattice models of random surfaces belonging to the EW universality class (e.g., the height function of the dimer model, the solid-on-solid model, and the interface of the Ising model) would be of significant interest. 
A modified setting of the DGFF that one could explore is the mixing time of random surface models in the presence of a hard floor (conditioned on being non-negative), inducing an entropic repulsion effect. While in dimension one, sharp mixing time bounds, including cutoff at $\frac{1}{\pi^2} n^2 \log n$ have been extended to this setting in~\cite{CMT-repulsion,Yang-repulsion}, in dimension two, the situation is more complicated. In fact, for certain integer-valued random surface models (e.g., the solid-on-solid model) the mixing time can slow down exponentially due to such a floor effect~\cite{CLMST}.

Another natural next step is to move beyond the 2D setting to the DGFF on more general families of graphs, in particular subsets of $\mathbb Z^d$ for $d\ge 3$. There, we do not expect the sticky coupling to coalesce in $o(n^2 \log n)$ time, and one would likely need to rely on the quenched Gaussianity of the interface together with the Fourier techniques hinted at above to demonstrate cutoff. 

Lastly, we mention that there is an entirely different class of random surface growth where the growth has a a non-linear dependence on the local gradient. This universality class is dictated by the stochastic PDE known as the \emph{KPZ equation}~\cite{KPZ} which one obtains by adding the non-linear term $|\nabla h(x,t)|^2$ to the right-hand side of~\eqref{eq:she}. In contrast to the SHE, the non-linearity is expected to give rise to non-Gaussian fluctuations. In one dimension, the KPZ equation has seen enormous attention and integrable features of certain models have led to remarkable progress (see~\cite{corwin2} for a survey). In the mixing time direction, cutoff was recently shown for the asymmetric exclusion process (a natural Markov chain belonging to the KPZ class) on the line segment in~\cite{Labbe-Lacoin-ASEP,Bufetov-ASEP}.
In higher dimensions much less is known, and in fact some recent works show that under certain renormalizations, the KPZ equation may even admit EW limits and Gaussian fluctuations in sub-critical regimes~\cite{Chatterjee-Dunlap,Caravenna-Rongfeng-Zygouras,Magnen-Unterberger}. 
Showing that the KPZ equation yields a non-Gaussian fluctuation theory under some renormalization in dimensions $d\ge 2$ remains a major open problem.
In light of the above advances, 
studying mixing times and cutoff in random surface growth models belonging to the KPZ universality class is an exciting avenue for future research.

\subsection*{Acknowledgements} The authors thank the anonymous referees for their careful reading and suggestions. The authors thank H.\ Lacoin and P.\ Diaconis for useful suggestions. S.G. was partially supported by NSF grant DMS-1855688, NSF Career grant DMS-1945172, and a Sloan Fellowship. R.G.\ thanks the Miller Institute for Basic Research in Science for its support.

	\section{A random walk representation of the DGFF evolution}\label{sec:random-walk-representation}
	In this section, we formally construct the random walk representation of the Glauber dynamics for the DGFF; this serves as a key source of regularity estimates for the height process over time, as alluded to in  Section~\ref{subsec:proof-ideas}.

	\subsection{Formal definition of the Glauber dynamics}
	We begin by giving a formal definition of the continuous-time Glauber dynamics for the DGFF on $\Lambda_n$. 

	\begin{definition}\label{def:surface-RW-coupling}
		Fix a realization of Poisson noise $\cQ  = \cQ_n$ on $\Lambda_n \times [0,t]$, enumerate the set of all times of $\cQ$ in increasing order as 
		\begin{align*}
		0<\sigma_1 < \sigma_2 < \cdots < \sigma_{N(t)}<t\,.
		\end{align*}
		and let $\nu_1,...,\nu_{N(t)}$ be the vertices at which these clock rings occur. 
		For convenience, let $\sigma_0 = 0$. Let $(G_1,...,G_{{N(t)}})$ be a sequence of i.i.d.\ standard normal random variables. 
		Given initial data $\bh(0) = (h(x,0))_{x\in \Lambda_n}$, construct $\bh^{\cQ}(t) = (h^{\mathcal Q}(x,t))_{x\in \Lambda_n}$  as follows. 
		\begin{itemize}
			\item For all $s\in [0, \sigma_1)$, let $h^{\mathcal Q}(x,s) = h(x,0)$ for all $x\in \Lambda_n$.
			\item For every $j\ge 1$, for all $s\in [\sigma_j, \sigma_{j+1})$, let 
			\begin{align*}
			h^{\cQ}(x, s) = \begin{cases}
			h^{\cQ}(x,\sigma_{j-1}) & x\ne \nu_{j} \\ 
			G_{j} + \frac{1}{4} \sum_{w\sim \nu_{j}} h^{\cQ}(w, \sigma_{j-1}) & x = \nu_{j} 
			\end{cases}\,.
			\end{align*}
		\end{itemize}
	\end{definition}
	
	\begin{remark}
		Notice that the update rule is the same as resampling the value of $h^{\cQ}$ at the site $\nu_j$ conditionally on its values on the neighbors of $\nu_j$. This is because the distribution of the DGFF at a vertex $v$, conditionally on the values of its neighbors is exactly distributed as a unit variance Gaussian whose mean is the average of its neighbors' values.
	\end{remark}
	
	\begin{remark}\label{rem:identity-coupling}
	Definition~\ref{def:surface-RW-coupling} naturally induces a coupling of DGFF dynamics chains with different initializations by use of the same Poisson update sequence $\cQ$ and the same Gaussian random variables $G_j$. This coupling will be called the \emph{identity coupling}, and can easily be seen to be a monotone coupling. 
	\end{remark}

	Let us pause to use the formal representation of the above definition to reason that it suffices for us to consider the DGFF evolution with identically zero boundary conditions, as alluded to in Remark~\ref{rem:bc}. The following shows that the dynamics with any other boundary conditions can be coupled to simply be a vertical shift of the one with the zero boundary data. For a function $\eta$ on $\partial \Lambda_n$, let $\Harm_\eta$ be the discrete harmonic extension of $\eta$, i.e., the unique function on $\Lambda_n \cup \partial \Lambda_n$ that takes values $\eta$ on $\partial \Lambda_n$ and is discrete harmonic on $\Lambda_n$. 
	\begin{lemma}\label{lem:boundary-conditions}
		Let $\bh$ be the Glauber dynamics with zero boundary conditions, and let $\bh_1$ be the Glauber dynamics with boundary conditions $\eta_1 \in \mathbb R^{\partial \Lambda_n}$. Initialize $\bh_1(0) = \Harm_{\eta_1} + \bh(0)$. Then, 
		\begin{align*}
		\bh_1(t) - \Harm_{\eta_1}  \stackrel{d}= \bh(t)\,.
		\end{align*}
	\end{lemma}
	
	\begin{proof}
		It suffices to show that the left and right-hand sides agree with probability one if they are constructed by Definition~\ref{def:surface-RW-coupling} using the same $\cQ$ and the same Gaussians $(G_j)_j$. Fix almost any $\cQ$, and enumerate the times of $\cQ$ in increasing order as 
		\begin{align*}
		0<\sigma_1<\sigma_2< \cdots <\sigma_{N(t)}<t\,.
		\end{align*}
		We will prove inductively over $j$ that 
		\begin{align*}
		\bh_1(\sigma_j) - \Harm_{\eta_1} = \bh(\sigma_j)\,.
		\end{align*}
		The base case of $\sigma_0:=0$ is evident by construction, as $\bh_1(0) - \Harm_{\eta_1} = \bh(0)$. Suppose now that for the first $j-1$ updates, the equality holds. Then, at time $\sigma_j$, suppose a vertex $v_j$ is updated. The equality on all vertices other than $v_j$ evidently carry over from $\sigma_{j-1}$. The relevant updates made at $v_j$ are as follows: 
		\begin{align*}
		\bh_1(v_j,\sigma_j) & = \frac{1}{4}\sum_{z\sim v_j} \bh_1(z,\sigma_{j-1}) + G_j\,, \qquad \mbox{and}\qquad \bh(v_j,\sigma_j)  = \frac{1}{4}\sum_{z\sim v_j} \bh(z,\sigma_{j-1}) + G_j\,.
		\end{align*}
		Then, by the inductive hypothesis,
		\begin{align*}
		\bh_1(v_j,\sigma_j) - \bh(v_j,\sigma_j) = \frac{1}{4}\sum_{z \sim v_j} \Harm_{\eta_1}(z)\,,
		\end{align*}
		which by harmonicity of $\Harm_{\eta_1}$, is exactly $\Harm_{\eta_1}(v_j)$. 
	\end{proof}
	
	\subsection{Random walk representation of the DGFF evolution}
	Given the construction of Definition~\ref{def:surface-RW-coupling}, we now wish to define a backwards-in-time random walk (BTRW) through a Poisson noise field, i.e., it jumps whenever it encounter a point in the latter, through which we will give an alternative representation of the surface evolution $\bh(t)$.

	Let $\cP_{\mathbb Z^2} = (p_{v,i})_{v\in \mathbb Z^2, i\ge 1}$ be a Poisson field in the space-time slab $\mathbb Z^2 \times (-\infty, 0]$, i.e., for each $v\in \mathbb Z^2$, assign an independent Poisson clock, with clock rings at $(t_{v,i})_{i\ge 1}$; then let $p_{v,i} = (v,-t_{v,i})\in \mathbb Z^2 \times (-\infty,0]$. For ease of notation, let $\cP_{n} = \cP_{\Lambda_n}$ be the restriction of $\cP_{\mathbb Z^2}$ to $\Lambda_n \times (-\infty,0]$, and thus let $\cP_\infty = \cP_{\mathbb Z^2}$. When understood from context, we will drop the $n$ notation so that $\cP = \cP_n$.

	\begin{definition}\label{def:backwards-rw}
		Fix a realization of the Poisson field $\cP_\infty$ such that {$(t_{v,i})_{v,i}$} are all distinct (this event has probability one). Then, we can easily identify the times $t\in \{t_{v,i}: v\in \mathbb Z^2, i\ge 1\}$ with the unique space-time point $p\in \cP_\infty$ such that $p$ is at time $-t$, and vice versa. For every $x\in \Lambda_n$, define the BTRW in $\cP= \cP_n$ started from $x$, denoted $(S_{-s}^{x,\cP})_{s \ge 0}$, as follows: 
		\begin{enumerate}
			\item Let $(\mathcal E^x_{j})_{j\ge 1}$ be a sequence of i.i.d.\ jumps uniform at random among $\{\pm e_1,...,\pm e_d\}$.   
			\item Construct the discrete-time random walk $(Y_k)_{k\ge 0}$ by initializing $Y_0 = x$ and for each $i\ge 1$, letting $Y_i = Y_{i-1} + \cE_i^x$. 
			\item Initialize $S_{-s}^{x,\cP} = Y_0$ for all $-s\in [-T^x_1,-T_0^x]$ where $-T^x_1 = -t_{x,1}$ and $T^x_0 = 0$.
			\item For each $i \ge 1$, let 
			$$S_{-s}^{x,\cP} = Y_i \qquad \mbox{for all $-s\in [-T_{i+1}^x, -T_{i}^x)$}\,,$$
			where 
			$$-T_{i+1}^x = \max \{-t< -T_{i}^x: -t\in \{-t_{Y_i,j}\}_{j\ge 1}\}\,.$$
			
		\end{enumerate}
	\end{definition}
	
	\begin{definition}
		We denote the random walk \emph{trajectory}  $(S_{-s}^{x,\cP})_{-s\in [-t,0]}$ by $S_{[-t,0]}^{x,\cP}$. We view this as a right-continuous function from $[-t,0] \to \mathbb Z^2$. (While right-continuity in time is more typical of forward-in-time random walks, we constructed the BTRW as right-continuous to match the right-continuity of the Glauber dynamics.) The BTRW trajectory is in direct correspondence with the sequence of Poisson points it collects $\{p\in \cP: p\in S_{[-t,0]}^{x,\cP}\}$ (and almost surely in direct correspondence with the sequence of times $\{-T_{i}^{x}\ge -t\}$). We therefore use $S_{[-t,0]}^{x,\cP}$ both to denote the function from $[-t,0] \to \mathbb Z^2$ and the set of Poisson points or times it passes through.  
	\end{definition}
	
	\begin{remark}
		Observe that using $\cP_n$ to generate the random walk is the same as absorbing the walk at $\partial \Lambda_n$, since as soon as it jumps onto a vertex of $\partial \Lambda_n$ it will never see another clock ring, and therefore never move again. 
	\end{remark}
	
	We refer the reader to Figure~\ref{fig:rw-rep} for a depiction of BTRW trajectories in~$\cP$. 
	
	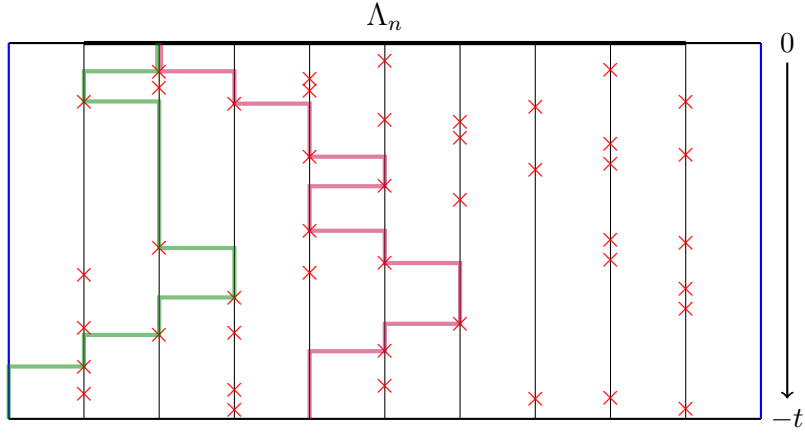
\begin{figure}
		\centering
		\begin{tikzpicture}
		\draw[ultra thick] (1,0)--(9,0);
		\draw[thick] (0,0)--(10,0);
		\foreach \i in {0,...,10}{\draw (\i,0)--(\i,-5);};
		\draw[thick] (0,-5)--(10,-5);
		\draw[ thick, color = blue] (0,0)--(0,-5);
		\draw[ thick, color = blue] (10,0)--(10,-5);
		
		\node[font = \large] at (5,.35) {$\Lambda_n$};
		\node (A0) at (10.35,0) {$0$};
		\node (At) at (10.35,-5) {$-t$};
		\draw[thick, ->] (A0)--(At);

		\tikzstyle{time-mark}=[color = red, opacity = 1]
		
		\foreach \i in {1,...,5}{\coordinate (A1\i) at (1,-{Mod(1.311^(10+\i),5)}) {}; \node[style = time-mark] at (A1\i) {$\times$};};
		\foreach \i in {1,...,4}{\coordinate (A2\i) at (2,-{Mod(1.414^(10+2*\i),5)}) {}; \node[style = time-mark] at (A2\i) {$\times$};};
		\foreach \i in {1,...,5}{\coordinate (A3\i) at (3,-{Mod(1.514^(9+2*\i),5)}) {}; \node[style = time-mark] at (A3\i) {$\times$};};
		\foreach \i in {1,...,5}{\coordinate (A4\i) at (4,-{Mod(1.624^(10+2*\i),5)}) {}; \node[style = time-mark] at (A4\i) {$\times$};};
		\foreach \i in {1,...,6}{\coordinate (A5\i) at (5,-{Mod(1.148^(10+\i),5)}) {}; \node[style = time-mark] at (A5\i) {$\times$};};
		\foreach \i in {1,...,4}{\coordinate (A6\i) at (6,-{Mod(1.244^(10+\i),5)}) {}; \node[style = time-mark] at (A6\i) {$\times$};};
		\foreach \i in {1,...,3}{\coordinate (A7\i) at (7,-{Mod(1.344^(10+\i),5)}) {}; \node[style = time-mark] at (A7\i) {$\times$};};
		\foreach \i in {1,...,6}{\coordinate (A8\i) at (8,-{Mod(1.777^(10+\i),5)}) {}; \node[style = time-mark] at (A8\i) {$\times$};};
		\foreach \i in {1,...,6}{\coordinate (A9\i) at (9,-{Mod(1.555^(9+\i),5)}) {}; \node[style = time-mark] at (A9\i) {$\times$};};
		
		\draw[color = green!50!black, opacity = .5, ultra thick] (1.975,0)-- (1.975,-{Mod(1.414^(10+2*3),5)}) -- ++(-.975,0) -- (A12) -- ++(1,0)--(A22) -- ++(1,0) --(A33)--++(-1,0)--(A21)--++(-1,0)--(A14)--++(-1,0)--(0,-5);
		
		\draw[color = purple, opacity = .5, ultra thick] (2.025,0)-- (2.025,-{Mod(1.414^(10+2*3),5)}) -- ++(.975,0) -- (A31)--++(1,0)--(A41)--++(1,0)--(A54)--++(-1,0)--(A42)--++(1,0)--(A55)--++(1,0)--(A62)--++(-1,0)--(A56)--++(-1,0)--(4,-5);
		\end{tikzpicture}
		\caption{The Poisson field $\cP_n$ on $\Lambda_n \times [-t,0]$ is depicted by the red $\times$-symbols. Two BTRW trajectories  through $\cP_n$ initialized from $x=2$ are depicted in red and green. Each point in $\cP_n$ is assigned an i.i.d.\ standard Gaussian $(Z_p)_{p\in \cP_n}$. The green trajectory hits $\partial \Lambda_n$ (blue) and gets absorbed there, so its contribution to $\bh^{\cQ_t}$ in Theorem~\ref{thm:rw-representation-exact-equality} is only in the Gaussians it collects. The red trajectory is in $\Lambda_n$ at time $-t$ and therefore additionally picks up the value of the initialization at the point $S_{-t}^{x,\cP}=4$ it ends up at.}
		\label{fig:rw-rep}
	\end{figure}

	The main result of this section is the following representation of the Glauber dynamics for the DGFF surface on $\Lambda_n$ in term of expectations of the above random walk trajectories through the Poisson noise field $\cP_n$. Assign every $p\in \cP_\infty$ an independent $\mathcal N(0,1)$ random-variable $Z_p$ alternatively denoted $Z_{-t}$ if $p$ is at time $-t$ (almost surely all points of $\cP_\infty$ are at distinct times). We will show that conditionally on the clock rings of the Glauber dynamics between times $[0,t]$, there is a corresponding Poisson noise field on $\Lambda_n \times [-t,0]$, in which the Glauber dynamics is expressed as averages of the Gaussians $(Z_p)$ collected by the BTRW trajectory at times $[-t,0]$.

	Let $\cQ$ be a Poisson noise field on $\mathbb Z^2 \times [0,t]$, and let $h^\cQ(t)$ be the Glauber dynamics whose sequence of clock rings is exactly $\cQ$. Let $\vartheta_r \cP$ be the unit intensity Poisson process on $\Lambda_n \times [0,r]$ obtained by time-shifting $\cP$ by $r$ and then restricting to non-negative times, i.e., it simply considers $\cP$ restricted to $\Lambda_n\times [-r,0]$ shifted in the time coordinate by $r$ to fit $\Lambda_n\times [0,r].$
	
	\begin{theorem}\label{thm:rw-representation-exact-equality}
		Fix almost any Poisson noise process $\cP = \cP_n$, and let $\cQ_t =  \vartheta_t\cP$. Let $(Z_p)_{p\in \cP}$ be i.i.d.\ standard Gaussian random variables. For any initialization $\bh(0)$, for all $t>0$, 
		\begin{align*}
		\bh^{\cQ_t}(t) \stackrel{d}= \bigg(\mathbb E_{\mathcal E^x} \Big[h\big(S_{-t}^{x,\cP},0\big) + \sum_{p \in \bS_{[-t,0]}^{x,\cP}} Z_p \Big]\bigg)_{x\in \Lambda_n}\,.
		\end{align*}
	\end{theorem}
	
	By averaging over $\cQ$ and $\cP$ and noting that the shift $\vartheta_{t}$ is a measure preserving transformation for the Poisson process, we obtain the following corollary. 
	
	\begin{corollary}\label{cor:random-walk-representation}
		For every initialization $\bh(0)$, for all $t>0$, we have 
		\begin{align*}
		\bh(t) \stackrel{d}= \Bigg(\mathbb E_{\mathcal E^x} \Big[h\big(S_{-t}^{x,\cP},0\big) + \sum_{p \in \bS_{[-t,0]}^{x,\cP}} Z_p \Big]\Bigg)_{x\in \Lambda_n}\,,
		\end{align*}
		where the randomness on the right-hand side is over both $\cP = \cP_n$ and $(Z_p)_{p\in \cP}$. 
	\end{corollary}
	
	\begin{proof}[\textbf{\emph{Proof of Theorem~\ref{thm:rw-representation-exact-equality}}}]
		The proof will follow from an explicit coupling between the process on the left-hand side as defined in Definition~\ref{def:surface-RW-coupling} and the process on the right-hand side, defined in Definition~\ref{def:backwards-rw}. Fix almost any $\cP$. Enumerate the set of all times in $\cP = \cP_n$ as 
		\begin{align*}
		0> -\tau_1 > -\tau_2> \cdots\,.
		\end{align*}
		For any $s$, let $N(s) = \max\{ j: -\tau_j >-s\}$. 
		We claim that for every $s$, for any initial data $\bh(0)$,
		\begin{align}\label{eq:rw-representation-exact-equality}
		h^{\vartheta_s \cP}(x,s) = \mathbb E_{\mathcal E^x} \Big[h\big(S_{-s}^{x,\cP},0\big) + \sum_{p \in \bS_{[-s,0]}^{x,\cP}} Z_p \Big] \qquad \mbox{for all $x\in \Lambda_n$}\,,
		\end{align}
		holds if $h^{\vartheta_{s}\cP}(x,s)$ is generated using Definition~\ref{def:surface-RW-coupling} with
		\begin{align}\label{eq:Gaussian-G-Z-coupling}
		G_j = Z_{(v_{N(s) + 1-j}, -\tau_{N(s)+ 1 -j})}\qquad \mbox{for all $j\le N(s)$}\,,
		\end{align}
		where $v_l$ is the location of the clock ring at time $-\tau_{l}$. This is of course a valid coupling since both the $(G_j)_j$ and $(Z_p)_{p}$ are sequences of i.i.d.\ standard Gaussians. The desired result of the theorem would follow from establishing~\eqref{eq:rw-representation-exact-equality} with the choice $s=t$ for any realization of $(Z_p)_{p\in \cP}$. In the rest of this proof, both $\cP$ and $(Z_p)_{p\in \cP}$ will be fixed, so all expectations should be understood to be only over the jump sequence $\cE^x$, and we drop that subscript from the notation.  
		
		By right-continuity of the quantities on either side of~\eqref{eq:rw-representation-exact-equality}, it suffices to prove that equality for all $s\notin \{\tau_j\}_j$. Towards that, divide $(-\infty,0]\setminus \{\tau_j\}_j$ into intervals $I_0 = (-\tau_1,0]$ and $I_k = (-\tau_{k+1},-\tau_k)$. We will prove inductively over $k$ that~\eqref{eq:rw-representation-exact-equality} holds for all $s\in I_k$ and all initial data $\bh(0)$.

		If $k=0$, then evidently for every $s\in I_0$, $\vartheta_s \cP$ will not contain any clock rings, and $h^{\vartheta_s \cP}(x,s) = h(x,0)$ for all $x$. On the right-hand side the random walk $\bS_{[-s,0]}^{x,\cP}$ will not contain any jumps from $x$, and the right-hand also yields $h(x,0)$ verifying the base case. 

		Now suppose that the equality of~\eqref{eq:rw-representation-exact-equality} holds for all $s\in \bigcup_{l <k} I_l$; we will show it holds for all $s\in I_k$. By the inductive hypothesis, we have for all $s\in I_{k-1}$  
		\begin{align*}
		h^{\vartheta_s \cP}(x,s) = \mathbb E \Big[h(S_{-s}^{x,\cP},0) + \sum_{p\in \bS_{[-s,0]}^{x,\cP}} Z_p\Big]\qquad \mbox{for all $x\in \Lambda_n$}\,.
		\end{align*}
		Now notice that for $r>s$, $\bh^{\vartheta_r\cP}(r)$ initialized from $\bh(0)$ can be generated by letting $$\bg_r = \bh^{\vartheta_r \cP}(r-s)\,,$$
		and then using $\bg_r$ as the initial data for a new Glauber dynamics chain $\bh^{\vartheta_s \cP}(s)$. By this reasoning, for all $r>s$, we have
		\begin{align*}
		h^{\vartheta_{r}\cP}(x,r) = \mathbb E\Big[g_r(S_{-s}^{x,\cP}) + \sum_{p\in \bS_{[-s,0]}^{x,\cP}} Z_p\Big]\qquad \mbox{for all $x\in \Lambda_n$}.
		\end{align*}
		Let $r\in I_k$, and note that $N(s)+1 = N(r) = k$. Since under the coupling for $h^{\vartheta_r \cP}$, we have  $G_{1} = Z_{(v_{k},-\tau_{k})}$, we get    
		\begin{align*}
		g_r(y) = \begin{cases}
		h(y,0) & y\ne v_k \\
		Z_{(v_k,-\tau_{k})} + \frac{1}{4} \sum_{w\sim v_{k}} h(w,0) & y = v_k
		\end{cases}\,.
		\end{align*}
		Plugging in the expression for $g_r(y)$ in the display above it, we see that 
		\begin{align}\label{eq:E-tilde-h-expansion}
		\mathbb E [ g_r(S_{-s}^{x,\cP})] &  = \sum_{y\in \Lambda_n} g_r(y) \mathbb P(S_{-s}^{x,\cP} = y) \nonumber \\ 
		& = \sum_{y\ne v_{k}} h (y,0) \mathbb P(S_{-s}^{x,\cP} = y) + \mathbb P(S_{-s}^{x,\cP} = v_{k}) Z_{(v_k,-\tau_{k})}   + \frac{1}{4} \mathbb P(S_{-s}^{x,\cP} = v_{k}) \sum_{w\sim v_k} h(w,0)\,. 
		\end{align}
		Now notice that if $S_{-s}^{x,\cP} =y$ for a $y\ne v_k$, then since the only clock ring in $\Lambda_n$ in the time interval $(-r,-s)$ occurs at the site $v_{k}$, we have that $S_{-r}^{x,\cP} = S_{-s}^{x,\cP}$. On the other hand, if $S_{-s}^{x,\cP} = v_k$, then the probability that $S_{-r}^{x,\cP} = y$ is $\frac{1}{4}$ if $y\sim v_k$, and zero otherwise. Therefore,  
		\begin{align*}
		\mathbb P(S_{-r}^{x,\cP} = y) = 
		\begin{cases}
		0 & y = v_{k} \\
		\mathbb P(S_{-s}^{x,\cP} = y) + \frac{1}{4}\mathbb  P(S_{-s}^{x,\cP} = v_k) & y\sim v_k \\ 
		\mathbb P(S_{-s}^{x,\cP} = y) & \mbox{else}
		\end{cases}\,.
		\end{align*}
		It is then evident that the first and third terms together in the expansion of~\eqref{eq:E-tilde-h-expansion} are exactly the quantity 
		$\mathbb E[h(S_{-r}^{x,\cP},0)]$. 
		It remains to show that  
		\begin{align*}
		\mathbb P(S_{-s}^{x,\cP} = v_k) Z_{(v_k,-\tau_k)} = \mathbb E\Big[ \sum_{p\in \bS_{[-r,0]}^{x,\cP}} Z_p - \sum_{p\in \bS_{[-s,0]}^{x,\cP}} Z_p\Big]\,.
		\end{align*}
		To see this, notice that the right-hand side is exactly 
		\begin{align*}
		\mathbb E\Big[ \sum_{p\in \bS_{[-r,-s]}^{x,\cP}} Z_p\Big] = \mathbb E\Big[ Z_{(v_k,-\tau_{k})} \mathbf 1\{S_{-s}^{x,\cP} =v_k\}\Big]\,,
		\end{align*}
		as the only clock ring in the interval $[-r,-s]$ is the one at $(v_k, -\tau_k)$. 
		Altogether, we have obtained the desired equality of~\eqref{eq:rw-representation-exact-equality}.   
	\end{proof}

	\section{Consequences of the random walk representation}
	Having established the BTRW representation of Theorem~\ref{thm:rw-representation-exact-equality}, we now derive some fundamental consequences of this representation. 
	
	Before we proceed, let us introduce a change of notation that will prove to be convenient. Because the process on the left-hand side of Theorem~\ref{thm:rw-representation-exact-equality} requires one to keep shifting $\cP$ by time $t$ to obtain $\cQ_t$, and this is inconvenient to think about, let us solely work with the process on the right-hand side. Abusing notation, for each $\cP = \cP_n$, define the random process
	\begin{align}\label{eq:h-P}
	\bh^\cP(t) = \big(h^\cP(x,t)\big)_{x\in \Lambda_n} = \bigg(\mathbb E_{\cE^x}\Big[h(S_{-t}^{x,\cP},0) + \sum_{p\in S_{[-t,0]}^{x,\cP}} Z_p\Big]\bigg)_{x\in \Lambda_n}\,.
	\end{align}
	Recall from Theorem~\ref{thm:rw-representation-exact-equality} that for any fixed $t$, there is a direct coupling of this process to $\bh^{\cQ_t}(t)$, where $\cQ_t = \vartheta_t \cP$, but this coupling depends on $t$. Thus, we prefer to use the distributional equality of Corollary~\ref{cor:random-walk-representation} to view the Glauber dynamics $\bh(t)$ at time $t$ as the average over $\cP$ of $\bh^{\cP}(t)$. 
	
	\subsection{Gaussianity, and mean and covariance expressions}
	We begin by establishing the Gaussianity of the process $\bh^{\cP}(t)$ and giving random-walk characterizations of its mean and covariance. 
	
	\begin{proposition}\label{prop:mean-and-variance-of-h}
		Fix almost any $\cP$. For every $\bh(0)$, and every $t>0$, the law of $\bh^{\cP}(t)$ is a multivariate Gaussian. Its mean and covariance can be expressed as follows: 
		\begin{align*}
		\mathbb E[h^{\cP}(x,t)] & = \mathbb E_{\mathcal E^x} \Big[h(S_{-t}^{x,\cP},0)\Big]\quad \mbox{for all $x\in \Lambda_n$}\,, \\ 
		\cov[h^{\cP}(x,t),  h^{\cP}(y,t)] & =  \mathbb E_{\mathcal E^x,\mathcal E^y}\Big[\big|\big\{p\in \bS_{[-t,0]}^{x,\cP} \cap \bS_{[-t,0]}^{y,\cP}\big\}\big|\Big]\quad \mbox{for all $x,y\in \Lambda_n$}\,.
		\end{align*}
		Here $\mathbb E_{\cE^x,\cE^y}$ is the expectation with respect to the product measure $\mathbb P_{\cE^x}\otimes \mathbb P_{\cE^y}$. 
	\end{proposition}
	
	\begin{proof}
		We start by expressing $\bh^{\cP}(t)$ as a linear transformation of a sequence of independent standard Gaussians $(Z_p)_{p\in \cP}$. First of all, for a fixed $\cP$, for each $t,$ let $N(t)$ be the total number of clock rings in the spacetime slab $\Lambda_n \times [-t,0]$. Let $-\tau_{N(t)} < - \tau_{N(t)- 1}<... <-\tau_1<0$ be these clock rings ordered in time, and recall that the $Z_p$s can be labeled according to which time $p$ corresponds to. Let $\vec{Z} = (Z_{-\tau_1},..., Z_{-\tau_{N(t)}})$ be this $N(t)\times 1$ standard Gaussian vector.   
		
		By construction, the $|\Lambda_n| \times 1$ random vector $\bh^{\cP}(t) = (h^{\cP}(x,t))_{x\in \Lambda_n}$ can be expressed as 
		\begin{align*}
		h^{\cP}(x,t) & = \mathbb E_{\cE^x}[h(S_{-t}^{x,\cP},0)] + \sum_{\bS'_{[-t,0]}}\Big[ \mathbb P_{\cE^x}(\bS_{[-t,0]}^{x,\cP} = \bS'_{[-t,0]}) \big[\sum_{j: -\tau_j \in \bS'_{[-t,0]}} Z_{-\tau_j}\big]\Big] \\ 
		& =  \mathbb E_{\cE^x}[h(S_{-t}^{x,\cP},0)] + \sum_{1\le j\le N(t)} Z_{-\tau_j} \Big[\sum_{\bS'_{[-t,0]}\ni -\tau_j} \mathbb P_{\cE^x}(\bS_{[-t,0]}^{x,\cP} = \bS'_{[-t,0]})\Big]\,.
		\end{align*}
		where the sums are over feasible BTRW trajectories $\bS'_{[-t,0]}$ starting from $x$ (naturally identified with the times of the Poisson points it collects). 
		As such, we can express $$\bh^{\cP}(t) = \bm_t^{\cP} + \bA_t^{\cP} \vec{Z}\,,$$
		where $\bm_t^{\cP} = (m_t^\cP(x))_{x\in \Lambda_n}$ is a $|\Lambda_n| \times 1$ vector and $\bA_t^{\cP}$ is a $|\Lambda_n| \times N(t)$ matrix given as follows: 
		\begin{align*}
		\bm_t^{\cP}(x) = \mathbb E_{\cE^x}[h(S_{-t}^{x,\cP},0)] \qquad \mbox{and}\qquad \bA_t^{\cP} (x,j) =  \mathbb P_{\cE^x}(-\tau_j \in \bS_{[-t,0]}^{x,\cP})\,.
		\end{align*}

		Observe that for fixed $\cP$, neither $\bm_t^{\cP}$ nor $\bA_t^\cP$ are random. 
		This expression for $\bh^{\mathcal P}(t)$ as a linear transformation of $\vec{Z}$ thus implies that the former is distributed as a Gaussian random vector having 
		\begin{align*}
		\mathbb E[h^{\cP}(t)]  = \bm_t^{\cP} \qquad \mbox{and covariance matrix} \qquad \bSigma_t^{\cP} = \bA_t^{\cP} (\bA_t^{\cP})^{\mathsf{T}}\,.
		\end{align*}
		(The latter expression uses the fact that the covariance matrix of $\vec{Z}$ is the identity.) 
		The desired expression for the mean therefore follows immediately from the above. The expression for the covariance can be seen by writing 

		\begin{align*}
		\bSigma_t^{\cP}(x,y) = \sum_j \mathbb P_{\cE^x}(-\tau_j\in \bS_{[-t,0]}^{x,\cP}) \mathbb P_{\cE^y}(-\tau_j\in \bS_{[-t,0]}^{y,\cP})\,.
		\end{align*}
		
		This is exactly the expected number of clock rings shared by the two walks, under a product measure over the jumps $\cE^x,\cE^y$ yielding the expression desired by the proposition. 
	\end{proof}
	
	Let $G_n(\cdot,\cdot)$ be the Green's function for the random walk on $\Lambda_n$ killed at $\partial \Lambda_n$, i.e., 
	\begin{align}\label{eq:Greens-function}
	G_n(x,y) = \mathbb E\Big[\sum_{k\ge 0} \mathbf 1\{Y^x_k= y\}\Big]\,,
	\end{align}
	where $(Y^x_k)_{k\ge 0}$ is the simple random walk on $\Lambda_n\cup \partial \Lambda_n$ killed upon hitting $\partial \Lambda_n$. For convenience, we drop $n$ from the notation when understood from context and denote $G_n$ by $G$.  
	Recall that $G(\cdot,\cdot)$ is the covariance matrix for the DGFF on $\Lambda_n$. We will use $\bH_t^\cP$ to denote the heat kernel matrix for the BTRW $S_{-t}^{\cdot,\cP}$, i.e., for fixed $\cP= \cP_n$, for every $x,y\in \Lambda_n$, 
	\begin{align}\label{eq:H-t-P}
	\bH_t^\cP(x,y) = \mathbb P_{\cE^x}(S_{-t}^{x,\cP} = y)\,.
	\end{align}
	While Proposition~\ref{prop:mean-and-variance-of-h} included explicit expressions for the mean and the covariance in terms of averages of BTRW trajectories, the following result compares the same to their equilibrium counterparts, namely $0$ and $\bG$.

	\begin{proposition}\label{prop:monotonicity-of-mean-variance}
		For almost every $\cP$, we have for every $x,y\in \Lambda_n$
		\begin{align*}
		\lim_{t\to\infty} \mathbb E[h^{\cP}(x,t)] = 0\qquad \mbox{and} \quad \lim_{t\to\infty} \cov[h^{\cP}(x,t)h^{\cP}(y,t)] = \bG(x,y)\,.
		\end{align*}
		Furthermore, for every $t>0$,
		\begin{align*}
		\Big(\bG(x,y) - \cov[h^{\cP}(x,t),h^{\cP}(y,t)]\Big)_{x,y}  =   \Big(\bH_t^{\cP} \bG (\bH_t^{\cP})^{\mathsf{T}}\Big)_{x,y}\,. 
		\end{align*}
	\end{proposition}

	\begin{proof}[\textbf{\emph{Proof of Proposition~\ref{prop:monotonicity-of-mean-variance}}}]
		Fix any two initializations $\bh_1(0)$, $\bh_2(0)$, and let $\bh_1(t)$, $\bh_2(t)$ be their respective time evolutions. Use the same Poisson noise field $\cP$ and the same Gaussians $(Z_p)_{p}$ for both dynamical realizations. Let $g(x,0) = h_2(x,0) - h_1(x,0)$ for all $x\in \Lambda_n$. By Proposition~\ref{prop:mean-and-variance-of-h} and linearity of expectation, 
		\begin{align*}
		\big|\mathbb E[h_2^{\cP}(x,t)]- \mathbb E[h_1^{\cP}(x,t)]\big| = \big|\mathbb E_{\cE^x}[g(S_{-t}^{x,\cP},0)]\big| \le \|g(\cdot,0)\|_\infty \mathbb P_{\cE^x}(S_{-t}^{x,\cP}\notin \partial \Lambda_n)\,.
		\end{align*}
		Now observe that as long as $\cP_{\{v\}}$ is infinite for every $v$ (an event which has probability one), the probability above goes to zero as $t\to\infty$ (recalling that the set $\partial \Lambda_n$ is absorbing). This implies the convergence of the expectation, since if $\bh_2(0) \equiv 0$, then $\mathbb E[\bh_2^\cP(t)] \equiv 0$ for all $t>0$. 
		
		Now consider the difference in covariances. For fixed $\cP$ and using the same $(Z_p)_{p\in \cP}$, by the representation of Theorem~\ref{thm:rw-representation-exact-equality}, we have 
		\begin{align*}
		h_2^{\cP}(x,t)  - \mathbb E[h_2(S_{-t}^{x,\cP},0)] = h_1^{\cP}(x,t) - \mathbb E[h_1(S_{-t}^{x,\cP},0)] = \mathbb E_{\cE^x} \Big[\sum_{p\in \bS_{[-t,0]}^{x,\cP}} Z_p\Big]\,.
		\end{align*}
		The right-hand side quantity is exactly $h_3^{\cP}(x,t)$ (viewed as a random variable in $(Z_p)_p$), initialized from $h_3(x,0) = 0$ for all $x$. We can then rewrite
		$$h_2^{\cP}(x,t) = \mathbb E_{\cE^x}[h_2(S_{-t}^{x,\cP},0)] + h_3^\cP(x,t)\,.$$
		Let $\bh_2(0)$ be drawn from $\pi$ (independently of $\cP$ and $(Z_p)$ and hence of $h_3^\cP(x,t)$). Under both the randomness of $\pi$ and the randomness of $(Z_p)_p$ (but conditionally on $\cP$), by definition of the Glauber dynamics, $\bh_2^\cP(t)$ is distributed as $\pi$ for all $t$. Therefore, we get for almost every fixed $\cP$, 
		\begin{align*}
		\bG(x,y) & = \cov[h_2^{\cP}(x,t), h_2^{\cP}(y,t)] \\
		& = \cov_\pi\big[\mathbb E_{\cE^x}[h_2(S_{-t}^{x,\cP},0)], \mathbb E_{\cE^y}[h_2(S_{-t}^{y,\cP},0)]\big] + \cov[h_3^{\cP}(x,t), h_3^{\cP}(y,t)] \\
		& = \cov_\pi\big[\mathbb E_{\cE^x}[h_2(S_{-t}^{x,\cP},0)], \mathbb E_{\cE^y}[h_2(S_{-t}^{y,\cP},0)]\big] + \cov[h_1^{\cP}(x,t), h_1^{\cP}(y,t)]\,,
		\end{align*}
		where we used the fact that a deterministic shift in the initialization does not change the covariance in $(Z_p)$ to change $h_3^\cP$ to $h_1^\cP$. Subtracting the second term from both sides, and then swapping expectations in the first term, we get 
		\begin{align*}
		\bG(x,y) - \cov[h_1^{\cP}(x,t), h_1^{\cP}(y,t)] & = \mathbb E_{\pi}\big[\mathbb E_{\cE^x,\cE^y}[h_2(S_{-t}^{x,\cP},0) h_2(S_{-t}^{y,\cP},0)]\big] \\ 
		& = \mathbb E_{\cE^x,\cE^y}[\bG({S_{-t}^{x,\cP},S_{-t}^{y,\cP})}]\,,
		\end{align*}
		where the expectation $\mathbb E_{\cE^x,\cE^y}$ is under the product distribution $\mathbb E_{\cE^x}$ and $\mathbb E_{\cE^y}$. This last expectation is then easily seen to be exactly the $x,y$'th entry of the matrix $\bH_{t}^{\cP}\bG (\bH_{t}^{\cP})^{\mathsf{T}}$. Notice that this quantity is non-negative since all entries of $\bG(x,y)$ are non-negative. The last thing to show is that this quantity goes to zero as $t\to\infty$. This follows by observing that 
		\begin{align*}
		\mathbb E_{\cE^x,\cE^y} [\bG({S_{-t}^{x,\cP},S_{-t}^{y,\cP})}] \le 
		\|\bG\|_\infty \mathbb P_{\cE^x}(S_{-t}^{x,\cP}\notin \partial \Lambda_n)\mathbb P_{\cE^y}(S_{-t}^{y,\cP}\notin \partial \Lambda_n)\,,
		\end{align*}
		and recalling that almost surely both probabilities on the right-hand side go to zero as $t\to\infty$. 
	\end{proof}
	
	\section{Exponential decay rate of the mean process}
	In the previous section, we expressed the mean process of $\bh$, when conditioned on $\cP$, in terms of its BTRW representation. If the initial data is uniform (e.g., the all-$n$ initialization), then the mean $\mathbb E[h^\cP(x,t)]$ is evidently governed by the probability that $S_{-t}^{x,\cP}$ has not been absorbed into $\partial\Lambda_n$---what we call the survival probability of the BTRW. 
	
	Harmonic measure estimates for the BTRW in a fixed Poisson noise (i.e., quenched) can be quite different from those in discrete or continuous time; for instance if at time $-t$, the most recent clock ring were at site $v$, the probability that $S_{-t}^{x,\cP} = v$ is zero. Nonetheless, the survival probability is a sufficiently averaged quantity that for most $\cP$ it behaves as it would for the usual discrete random walk, for which the spectral theory, and therefore the quenched survival probability after $t$ steps, can be sharply characterized on the grid graph. That is the aim of this section.

	In what follows, the randomness will often come from the jump sequences $(\cE^x)_{x\in \Lambda_n}$ of the BTRW, with $\cP_\infty$ fixed. We will denote probabilities only over this randomness by $\mathbb P_{\cE}$. 
	The use of $\mathbb P$ in the context of the BTRW will be over both sources of randomness, $\cP_\infty$ and $(\cE^x)$.

	\medskip
	\noindent \textbf{Other notational disclaimers.} Throughout the rest of the paper, all constants in $O(1), \Omega(1)$ and $o(1)$ etc should be understood to hold with constants that are independent of $n$. In particular, all statements that follow should be understood to hold uniformly over $n$ sufficiently large. For ease of notation, we will use $c,C$ to denote the existence of constants independent of $n$ for which the relevant claim holds; these letters can denote different constants from line to line. Finally, for readability, we will ignore rounding issues and integer effects as it will be clear how to handle with minimal modifications.

	\subsection{Preliminaries for discrete-time random walks in {$\Lambda_n$}}\label{subsec:random-walk-prelims}
	We begin by collecting (standard) preliminaries on the spectrum of the Laplacian on $\Lambda_n$ and the harmonic measure and Green's function of the discrete-time simple random walk on~$\Lambda_n$. We refer to~\cite{Lawler-RW-book} for more details, and to Appendix~\ref{sec:appendix} for precise references for and proofs of these facts.
	
	Consider the (normalized) discrete Laplace operator $\Delta_n$ on $\Lambda_n$, given by the $|\Lambda_n|\times |\Lambda_n|$ matrix 
	\begin{align}\label{eq:Laplacian}
	\Delta_n(x,y) = \begin{cases}
	-1 & x=y \\
	\frac{1}{4} & x \sim y \\
	0 & \mbox{else}
	\end{cases}\,.
	\end{align}
	Let $\lambda_\one$ be the smallest eigenvalue of $-\Delta_n$, and let $\varphi_\one$ denote the corresponding normalized eigenvector. Then 
	\begin{align}\label{eq:top-eigenvalue}
	\lambda_\one = 1-\cos\Big(\frac{\pi}{n}\Big)\,,
	\end{align}
	and for $x= (x_1,x_2) \in \Lambda_n$, 
	\begin{align}\label{eq:top-eigenvector}
	\varphi_\one(x) = \prod_{j=1,2} \bigg(\sqrt{\frac{2}{n}} \sin\Big(\frac{x_j \pi}{n}\Big)\bigg)\,.
	\end{align}
	With that, we can also define a different normalization of this top eigenvector:  
	$$
	 \widehat \varphi_\one(x) := \varphi_\one(x) \sum_{y\in \Lambda_n} \varphi_\one(y)\,.
	 $$
	\begin{fact}\label{fact:order-of-magnitudes}
		For $\lambda_\one$ and $\varphi_\one$ as above, $\lambda_\one = \frac{\pi^2}{2n^2} + O(n^{-4})$ and $\max_{x\in \Lambda_n} \varphi_\one(x) = O(n^{-1})$. Further,  $\max_{x\in \Lambda_n}\widehat \varphi_\one(x) = O(1)$ and $\min_{x:d(x,\partial \Lambda_n) \ge n/4}\widehat \varphi_\one(x) = \Omega(1)$. 
	\end{fact}
	Let $(Y^x_k)_{k\ge 0}$ be the discrete-time simple random walk on $\Lambda_n$ that is killed as soon as it hits $\partial \Lambda_n$. Notice that if $P$ is the transition matrix for $Y_k$, then $-\Delta_n = I-P$. Using this relation, we obtain the following sharp bound on the survival probability of $(Y_k^x)_{k\ge 0}$ in $\Lambda_n$.
	
	\begin{lemma}\label{lem:survival-probability-srw}
		For all $k$ such that $n^2 = o(k)$, and all $x\in \Lambda_n$, we have 
		\begin{align*}
			\mathbb P(Y_k^x\notin \partial \Lambda_n) = \widehat \varphi_\one(x) (1-\lambda_\one)^k + o(e^{ - \lambda_\one k})\,.
		\end{align*}
		If $(\widetilde Y_t)_{t\ge 0}$ is the rate-1 continuous time version of $(Y_k)$, then for all $x\in \Lambda_n$ and $t$ such that $n^2 = o(t)$, 
		\begin{align*}
			\mathbb P(\widetilde Y_t^x \notin \partial \Lambda_n)  = \widehat \varphi_\one(x) e^{ - \lambda_\one t} + o(e^{ - \lambda_\one t})\,.
		\end{align*}
	\end{lemma}
	We also recall the size of the Green's function $G= G_n$ on $\Lambda_n$ which one can find, for instance, in Sections~4 and 8 of~\cite{Lawler-RW-book}. 
	
	\begin{fact}\label{fact:green-function-srw}
		For every $x,y \in \Lambda_n$, we have  
		\begin{align*}
		    G(x,y) \le C\log \Big(\frac{n}{|x-y|\vee 1}\Big)\,.
		\end{align*}
	\end{fact}

	\subsection{Number of clock rings collected by the random walk in $\cP$}
	The first estimates we require are on the number of Poisson points picked up, i.e., the number of jumps made, by the BTRW. 
	Denote the number of such jumps by time $-t$ as
	\begin{align*}
	N_{-t}^{x,\cP} = \Big|\Big\{s\in \cP: s\in \bS_{[-t,0]}^{x,\cP}\Big\}\Big|\,.
	\end{align*}
	
	\begin{lemma}\label{lem:number-of-clock-rings}
		Fix $t>0$ and $r>0$. With probability $1-{Ce^{ - r/C}}$, the Poisson field $\cP_\infty = \cP_{\mathbb Z^2}$ is such that,  
		\begin{align*}
		\mathbb P_\cE \Big( N_{-t}^{x,\cP_\infty} \notin \big(t - r \sqrt{t}, t+ r \sqrt t\big)  \mbox{ for some $x\in \Lambda_n$}\Big) \le  C n^2 e^{ - (r^2 \wedge r\sqrt{t})/C}\,.
		\end{align*}
	\end{lemma}
	
	\begin{proof}
		We will first bound the probability averaged over both the Poisson field and the random-walk jumps and then simply use Markov's inequality. Namely, we first establish that for every $x\in \Lambda_n$,
		\begin{align}\label{wts:number-of-jumps}
		\mathbb P\Big((\cP_\infty,\cE^x): N_{-t}^{x,\cP_\infty}\notin (t-r\sqrt{t}, t+ r\sqrt{t})\Big)\le C e^{ - (r^2 \wedge r\sqrt{t})/C}\,.
		\end{align}
		This will follow from the observation that the law of $\bS_{[-t,0]}^{x,\cP}$, when averaged over $\cP$, is simply that of a standard continuous-time random walk in $\Lambda_n$. Formally we will construct a coupling of the clock rings the random walk collects, and a Poisson process of intensity 1 on $[0,\infty)$  denoted $\mathcal T = (\mathcal T_1,\cT_2,...)$. 
		We claim that the law of the sequence $(T_1^x, T_2^x,...)$ of times of the Poisson points collected by the random walk $S^{x,\cP_\infty}$ is identical to the law of $(\cT_{1},\cT_2,...)$. To see this, let $\cF_t$ denote the filtration generated by $\cP_{\infty,t}$ together with the trajectory $S_{[-t,0]}^{x,\cP_\infty}$, i.e., the restriction of $\cP_\infty$ to times in $[-t,0]$ and the jumps $\cE^x$ that are used up to time $-t$. 
		Then observe that conditionally on $\cF_{T_1^x}$, $T_{j}^x - T_{j-1}^x$, is, by the memoryless property of exponential random variables (as well as the independence of the Poisson processes on the different columns $\{v\} \times (-\infty,0]$ and independence from $\cE^x$), simply distributed as an Exponential(1), which is the same as the distribution of $\cT_j -\cT_{j-1}$.  
		
		With that equality in distribution in hand, $N_{-t}^{x,\cP_\infty}$ is equal in distribution to the number of elements of $\cT$ that are smaller than $t$, i..e., the number of points of an intensity 1 Poisson process in an interval $[0,t]$ which is a $\mbox{Pois}(t)$ random variable.  Thus \eqref{wts:number-of-jumps} follows by the Chernoff bound, 
		\begin{align*}
		\mathbb P\big(\mbox{Pois}(t) \notin (t-r\sqrt{t}, t+r \sqrt{t})\big) \le \begin{cases}
		Ce^{ - r^2/C} & r\le \sqrt{t} \\ 
		Ce^{- r\sqrt{t}/C} & r>\sqrt{t}
		\end{cases}\,.
		\end{align*}
		By a union bound, 
		\begin{align*}
		\mathbb P\Big((\cP_\infty,(\cE^x)_x): \bigcup_{x\in \Lambda_n} \big\{N_{-t}^{x,\cP_\infty}\notin (t-r\sqrt{t}, t+ r\sqrt{t})\big\}\Big)\le C n^2 e^{ - (r^2 \wedge r\sqrt{t})/C}\,.
		\end{align*}Now using Markov's inequality, we find that 
		\begin{align*}
		\mathbb P\Big(\cP_\infty: \mathbb P\Big( \bigcup_{x\in \Lambda_n} N_{-t}^{x,\cP_\infty}\notin (t-r\sqrt{t}, & t+r\sqrt{t})\Big)  >\sqrt{C} n^2 e^{ - (r^2 \wedge r\sqrt{t})/2C} \Big) \\
		& \le \frac{C n^2 e^{- (r^2 \wedge r\sqrt{t})/C}}{\sqrt{C}n^2 e^{ - (r^2 \wedge r\sqrt{t})/2C}} \le \sqrt{C}e^{-(r^2 \wedge r\sqrt{t})/2C}\,.
		\end{align*}
		This implies the desired for some other choice of constant $C$. 
	\end{proof}

	The following lemma establishes that the random-walk through the Poisson noise field $\cP_n$ is a time-change of a discrete-time simple random walk, but where the time-change is trajectory dependent. The proof is immediate by taking the coupling of the internal clock of the random walk in Poisson noise, to an independent rate 1 Poisson clock, and then using the same sequence of jumps $\cE^x$ to generate the coupled simple random walk.  
	
	\begin{lemma}\label{lem:time-change-of-srw}
		Let $Y^x_k$ be the discrete-time simple random walk started from $x$ with jump sequence given by $\cE^x$. Let $\tau_0(k)= \inf \{t: N_{-t}^{x,\cP} = k\}$. Then for any $\cP$, we have 
		\begin{align*}
		(S_{-\tau_0(k)}^{x,\cP})_{k\ge 0}  = (Y_{k}^x)_{k\ge 0}\,.
		\end{align*}
		In particular, for every $k$, we have $\mathbb P_\cE (S_{-\tau_0(k)}^{x,\cP}\in \cdot) = \mathbb P(Y_k^x\in \cdot)$. 
	\end{lemma}

	\subsection{Survival probability of random walk in $\cP$}
	In this subsection, we use the bound on the number of steps taken by the random walk in the Poisson noise field from Lemma~\ref{lem:number-of-clock-rings} to prove the following sharp bound on the absorption probability of $S_{-t}^{x,\cP}$. Recall the top eigenvalue $\lambda_\one$ and corresponding eigenvector $\varphi_\one$ of the Laplacian on $\Lambda_n$ from~\eqref{eq:top-eigenvalue}--\eqref{eq:top-eigenvector}, and recall their orders of magnitude from Fact~\ref{fact:order-of-magnitudes}.

	\begin{proposition}\label{prop:mean-of-h}
		Fix any $t= \Theta(n^2 \log n)$. With probability $1-o(n^{-4})$, $\cP_\infty$ is such that for any $x\in \Lambda_n$,  
		\begin{align*}
		\mathbb P_\cE \big(S_{-t}^{x,\cP}\notin \partial \Lambda_n\big)  = (\widehat \varphi_\one(x) +o(1)) e^{ - \lambda_\one t}\,,
		\end{align*}
		recalling that  $\widehat \varphi_\one(x) = \varphi_\one(x) \sum_{y\in \Lambda_n} \varphi_\one(y)$. 
	\end{proposition}
	
	\begin{proof}
		For fixed $\cP_\infty$, define the event 
		\begin{align*}
		E_{t,x,r}^{\cP_\infty} = \big\{ N_{-t}^{x,\cP_\infty} \in (t-r\sqrt{t}, t+r\sqrt{t})\big\}\,,
		\end{align*}
		and let 
		\begin{align*}
		\Gamma = \Big\{\cP_\infty:  \mathbb P_\cE \big( \bigcup_{x\in \Lambda_n} (E_{t,x,r}^{\cP_\infty})^c\big)\le Cn^2 e^{ - r/C} \Big\}\,.
		\end{align*}
		Let $r= (\log n)^2$, and notice that by Lemma~\ref{lem:number-of-clock-rings}, $\Gamma$ has probability $1-o(n^{-4})$. Now fix any $\cP_\infty\in \Gamma$, and notice that since $\partial \Lambda_n$ is absorbing for the random walk in $\cP$, 
		\begin{align*}
		\mathbb P_\cE \big( S_{-t}^{x,\cP}\notin \partial \Lambda_n\big)  \le \mathbb P_\cE \big(\bS_{[-t,0]}^{x,\cP_\infty} \cap  \partial \Lambda_n  = \emptyset, E_{t,x,r}^{\cP_\infty}\big) + \mathbb P_\cE ((E_{t,x,r}^{\cP_\infty})^c)\,.
		\end{align*}
		The second term on the right-hand side is at most $Cn^2 e^{- r/C}$ by definition of $\Gamma$. Now consider the first term on the right-hand side and observe that the event $\{\bS_{[-t,0]}^{x,\cP_\infty}\cap \partial \Lambda_n = \emptyset\}$ is decreasing in $t$. Thus, we can maximize its probability by replacing $t$ with the random time $$\tau_0 (t-r\sqrt{t})= \inf \{s: N_{-s}^{x,\cP_\infty} = t-r \sqrt{t}\}$$
		(as on the event $E_{x,r,t}^{\cP_\infty}$ we necessarily have $\tau_{0}(t-r\sqrt{t})<t$). In particular, 
		\begin{align*}
		\mathbb P_\cE \big(\bS_{[-t,0]}^{x,\cP_\infty}\cap \partial \Lambda_n = \emptyset, E_{t,x,r}^{\cP_\infty}\big) & \le \mathbb P_\cE \big(\bS_{[-\tau_0(t-r\sqrt{t}),0]}^{x,\cP_\infty} \cap \partial \Lambda_n= \emptyset\big) \\ 
		& =  \mathbb P_\cE \big( S_{-\tau_0(t-r\sqrt{t})}^{x,\cP} \notin \partial \Lambda_n\big)\,.
		\end{align*}
		Recall from Lemma~\ref{lem:time-change-of-srw} that the distribution of $S_{-\tau_0(t-r\sqrt{t})}^{x,\cP}$ is exactly that of $t-r\sqrt{t}$ steps of a simple (discrete-time) random walk $(Y_k^x)_{k\ge 0}$ absorbed at $\partial \Lambda_n$. Then, by Lemma~\ref{lem:survival-probability-srw}, if we let $r= (\log n)^2$ and $t= O(n^2\log n)$ so that $n^2 = o(t - r\sqrt{t})$, we have 
		\begin{align*}
		\mathbb P_\cE \big(S_{-\tau_{0}(t-r\sqrt{t})}^{x,\cP}\notin \partial \Lambda_n\big)  = \mathbb P (Y_{t-r\sqrt{t}}^x\notin \partial \Lambda_n) = (\widehat \varphi_\one (x) +o(1)) e^{ - \lambda_\one (t- r\sqrt{t})}\,.
		\end{align*}
		At this point using the fact that $\lambda_\one = \Theta(n^{-2})$ and $\widehat \varphi_\one(x) = O(1)$, the above yields the desired
		\begin{align*}
		\mathbb P_\cE \big( S_{-t}^{x,\cP}\notin \partial \Lambda_n\big) & \le (\widehat \varphi_\one(x)+o(1)) e^{ - \lambda_\one (t-r\sqrt{t})} + Cn^2 e^{ - r/C} \\
		& \le (\widehat \varphi_\one(x)+o(1)) e^{ - \lambda_\one t}\,.
		\end{align*}
		
		Turning to the matching lower bound, we have 
		\begin{align*}
		\mathbb P_\cE \big( S_{-t}^{x,\cP}\notin \partial \Lambda_n\big) & \ge \mathbb P_\cE \big(\bS_{[-t,0]}^{x,\cP_\infty} \cap  \partial \Lambda_n  = \emptyset, E_{t,x,r}^{\cP_\infty}\big) \\
		& \ge \mathbb P_\cE \big(S_{-\tau_0(t+r\sqrt{t})}^{x,\cP}\notin \partial\Lambda_n = \emptyset, E_{t,x,r}^{\cP_\infty}\big)\,,
		\end{align*}
		using the fact that on $E_{t,x,r}^{\cP_\infty}$,  we have $\tau_0(t+r\sqrt{t}) > t$ (and the event $S_{[-t,0]}^{x,\cP_\infty}\cap \partial \Lambda_n = \emptyset$ is decreasing in $t$). Fixing any $\cP_\infty\in \Gamma$, this is in turn at least 
		\begin{align*}
		\mathbb P_\cE \big(S_{-\tau_0(t+r\sqrt{t})}^{x,\cP}\notin \partial\Lambda_n = \emptyset\big)- \mathbb P_\cE ((E_{t,x,r}^{\cP_\infty})^c) \ge \mathbb P (Y_{t+r\sqrt{t}}^x \notin \partial \Lambda_n) - Cn^2 e^{ - r/C}\,.
		\end{align*}
		By Lemma~\ref{lem:survival-probability-srw},  and the choices $r = (\log n)^2$ and $ t= O(n^2 \log n)$, this is at least 
		\begin{align*}
		(\widehat \varphi_\one(x) -o(1)) e^{ - \lambda_\one (t+r\sqrt{t})} - Cn^2 e^{ - r/C}\,,
		\end{align*}
		which,
		using the choices of $r,t$ and the fact that $\lambda_\one = \Theta(n^{-2})$ is $(\widehat \varphi_\one(x)-o(1)) e^{ - \lambda_\one t}$ as desired. 
	\end{proof}
	
	\subsection{Annealed survival probability estimates}
	We conclude this section by using the above quenched estimates to also obtain an annealed estimate on the mean decay for the DGFF Glauber dynamics for arbitrary initialization. 
	
	\begin{lemma}\label{lem:annealed-expectation-against-random-walk}
	For any non-negative function $f:\Lambda_n \to\mathbb R_+$, for every $t\ge n^2$ we have
	\begin{align*}
	    \mathbb E_{\cP}\Big[\sum_{x\in \Lambda_n} \mathbb E_{\cE}\big[f(S_{-t}^{x,\cP})]\Big] \le Ce^{ - \lambda_\one (t-n^2)} \sum_{x\in \Lambda_n} f(x)\,.
	\end{align*}
	where $\mathbb E_{\cP}$ denotes expectation with respect to the Poisson update sequence. 
	\end{lemma}
	
	\begin{proof}
	    By linearity of expectation, the left-hand side is 
	    \begin{align*}
	       \mathbb E_{\cP}\Big[\sum_{x\in \Lambda_n} \mathbb E_{\cE}\big[f(S_{-t}^{x,\cP})]\Big] = \sum_{x\in \Lambda_n} \sum_{y\in \Lambda_n} f(y) \mathbb E_{\cP} \mathbb P_{\cE}(S_{-t}^{x,\cP}= y)\,.
	    \end{align*}
	    The expectation over $\cP$ of the probability $\mathbb P_{\cE}(S_{-t}^{x,\cP} =  y)$ is, by the reasoning presented in Lemma~\ref{lem:number-of-clock-rings}, exactly the probability for a standard continuous-time random walk $(\widetilde Y_t^x)_{t\ge 0}$ with rate-one Poisson jump times and absorbed at $\partial \Lambda_n$. It therefore suffices to show that 
	    \begin{align}\label{eq:nts-one-point-hitting-probability}
	        \mathbb P_{\cE}(\widetilde Y_t^x =y) \le \frac{C}{n^2}e^{ - \lambda_\one t}\,.
	    \end{align}
	    To see this, observe that by the Markov property, we can bound the left-hand side as 
	    \begin{align*}
	        \mathbb P_{\cE}(\widetilde Y_{t-n^2}^x \notin \partial \Lambda_n ) \max_{z\in \Lambda_n} \mathbb P_{\cE}(\widetilde Y_{n^2}^z = y)\,.
	    \end{align*}
	    The first term here is $\widehat \varphi_\one (x) (1+o(1)) e^{ - \lambda_\one (t-n^2)}$ per Lemma~\ref{lem:survival-probability-srw}. The second term is upper bounded by, say, $\mathbb P_{\cE}(\bar Y_{n^2}^z = y)$ where $(\bar Y_t^z)$ is not killed at $\partial \Lambda_n$; by a local central limit theorem, this probability is at most $\frac{C}{n^2}$ for some universal constant $C$ as claimed. 
	\end{proof}

	\section{The volume supermartingale}\label{sec:volume-supermartingale}
	The bounds of Proposition~\ref{prop:mean-of-h} and Lemma~\ref{lem:annealed-expectation-against-random-walk} imply an exponential decay with rate $\Theta(n^{-2})$ for the mean of the DGFF dynamics per Proposition~\ref{prop:mean-and-variance-of-h}. This would already be essentially sufficient to deduce an $O(n^2 \log n)$ upper bound on the mixing time. Indeed this follows from what is essentially a union bound on the probability of not coupling in one sweep (a time period in which every site gets updated at least once) after the mean $\mathbb E[\bh(t)]$ is $o(n^{-2})$ everywhere. However, this would not attain the correct mixing time of $t_\star+o(n^2 \log n)$, which corresponds to the time when this mean process is simply $o(1)$ everywhere. In order to show the right mixing time upper bound, we switch to a  martingale  based argument building on that appearing  in~\cite{CLL-mixing-simplex,CLL-nabla-phi}.
	
The argument aims to couple two DGFF chains, one initialized from the maximal, all-$n$, configuration, with one initialized from some $\bg$ for some arbitrary $\bg$ having $\|\bg\|_\infty \le n$. Since in this section, we will be tracking the evolution of these chains simultaneously, and aiming to couple them to one another, we emphasize the initialization dependence in our notation, using $\bh^\bn$ to indicate the former chain, and $\bh^\bg$ to indicate the latter. 
	The fundamental quantity in the argument is the \emph{volume {supermartingale}} given by
	\begin{align}\label{eq:V-t}
	    V_t = \sum_{x\in \Lambda_n} dh_x(t)\,, \qquad \mbox{where} \qquad  dh_x(t) = h^{\vec{n}}(x,t) - h^{\bg}(x,t)\,.
	\end{align}
    The terminology of calling this process a \emph{supermartingale} is justified by Claim~\ref{clm:V-t-supermartingale}. 
    
    For each $\bg$, we construct a coupling of $h^{\bn}(t)$ and $h^\bg(t)$ such that $dh_x(t) \ge 0$ holds for all $x\in \Lambda_n$ and all $t\ge 0$, and such that at some $t = t_\star + O(n^2 \log \log n)$, we have
    \begin{align*}
        \mathbb P(V_t \ne 0) = o(1)\,.
    \end{align*}
    This will straightforwardly imply that the mixing time maximized over all initializatons $\|\bg\|_\infty \le n$ is at most $t_\star + O(n^2 \log \log n)$ as claimed in Theorem~\ref{thm:main-revised}. 
    
    Our approach to proving this is to stitch together two different couplings of the dynamics that preserve the monotonicity of the dynamics. The first of these couplings is the identity coupling defined by Remark~\ref{rem:identity-coupling}. This will be used for a time of $t_\star + C_1 n^2 \log \log n$ to get the volume supermartingale down to a size of $o(n^2(\log n)^{-5})$, corresponding to an average discrepancy of $o((\log n)^{-5})$ per site.
    
    \begin{lemma}\label{lem:stage-1}
    Under the identity coupling, 
    \begin{align*}
        \mathbb P\big( dh_x(t_\star + sn^2) \ge Ce^{-\pi^2 s/2} \mbox{ for some $x\in \Lambda_n$}\big) = o(1)\,.
    \end{align*}
    In particular, we have 
    \begin{align*}
        \mathbb P\big(V_{t_\star + sn^2}\le n^2 e^{ - \pi^2 s/2}) = 1-o(1)\,.
    \end{align*}
    \end{lemma}

    However, given the continuous nature of the state space, the identity coupling will never actually result in coalescence of the two Markov chain realizations. Moreover, even getting $V_t$ to $o(1)$ scales so that the \emph{total} discrepancy is truly negligible, takes a further $\Theta(n^2 \log n)$ time under that coupling, destroying any chance of proving cutoff. To address this, we use a second coupling, called the sticky coupling, that will be used to make the two chains coalesce in only a further $O(n^2 \log \log n)$ time. 
    
    \begin{definition}\label{def:two-stage-coupling}
    For two chains ${\bh}^\bn$ and ${\bh}^\bg$ with $\|\bg\|_\infty \le n$, the \emph{two-stage coupling} of $(\bh^\bn(t),\bh^\bg(t))_{t\ge 0}$ is the following.
    \begin{enumerate}
        \item For $T_0 := t_\star + \frac{12}{\pi^2}n^2 \log \log n$, couple $(\bh^\bn(t),\bh^\bg(t))_{t\in [0,T_0]}$ using the identity coupling. 
        \item Then, for all $t\ge T_0$, evolve $(\bh^\bn(t),\bh^\bg(t))$ using the sticky coupling which we will define shortly in Definition~\ref{def:DGFF-sticky-coupling}. 
    \end{enumerate}
    \end{definition}
    
    With the two-stage coupling defined, the main result of this section is the following. 
    
    \begin{proposition}\label{prop:V-t-down-to-polylog-scales}
        There exists $C$ such that the following holds. For every $\bg$ such that $\|\bg\|_\infty \le n$, under the two-stage coupling, 
        \begin{align*}
            \mathbb P\big( V_{t_\star + Cn^2 \log \log n} \le (\log n)^{5}\big) = 1-o(1)\,.
        \end{align*}
    \end{proposition}

    Note that there is still a remaining piece of the argument to get $V_t$ to actually hit zero. That step is reasonably straightforward and we will return to it in the following section; the remainder of this section will be dedicated to introducing the \emph{sticky coupling} and then establishing Proposition~\ref{prop:V-t-down-to-polylog-scales}.

    \subsection{The sticky coupling}
    We begin by describing the \emph{sticky coupling} of two Gaussian random variables, which is an \emph{optimal coupling} in the sense that it maximizes the probability that the two random variables agree.  
    
       \begin{definition}
       For two Gaussian random variables $X\sim \cN(-\mu,1)$ and $Y\sim \cN(\mu,1)$, the sticky coupling $\mathsf{P}_\mu$ is given as follows. Define the sub-probability measures 
    \begin{align}\label{eq:nu-1}
        d\nu_0^\mu(x) = (2\pi)^{-1/2} \big(e^{ - (x-\mu)^2/2}\wedge e^{ - (x+\mu)^2/2}\big)dx\,,
    \end{align}
    and 
    \begin{align}\label{eq:nu-2-3}
        d\nu_1^\mu(x)= (2\pi)^{-1/2}\big(e^{ - (x-\mu)^2/2} - e^{ - (x+\mu)^2/2}\big)\one_{\{x\ge 0\}}dx\,.
    \end{align}
    (Note that $d\nu^\mu_0 + d\nu^\mu_1$ is the law of $\cN(\mu,1)$.) 
    Then draw $(X,Y)\sim \mathsf{P}_\mu$ as follows: 
        \begin{align*}
            \begin{cases}
                    X=Y \sim \nu_0^\mu & \quad \mbox{w.prob. } \nu_0^\mu(\mathbb R) \\
                    -X = Y \sim \nu_1^\mu & \quad \mbox{w.prob. } \nu_1^\mu(\mathbb R)
            \end{cases}\,.
        \end{align*}
       \end{definition}
        It can be easily seen both that this coupling is monotone in the sense that $Y\ge X$, and that this coupling is an \emph{optimal} coupling of two Gaussians, i.e., $\mathbb P(X\ne Y) = \|N(-\mu,1) - N(\mu,1)\|_\tv$, which we recall is exactly $\erf(\mu/\sqrt{2})$. (Actually the details of how $(X,Y)$ are coupled on the complement of $\nu_0^\mu$ is not important but the choice we make is for concreteness.)
        Further, the above can be immediately generalized to the following.
        
        \begin{definition}\label{def:Gaussian-sticky-coupling}
        The sticky coupling $\mathsf{P}$ between two Gaussian random variables $Z_a \sim N(a,1)$ and $Z_b\sim N(b,1)$ for $a\le b$ is defined by taking $(X,Y)\sim \mathsf{P}_{(b-a)/2}$ and drawing $Z_a = X+\frac{a+b}{2}$ and $Z_b = Y+\frac{a+b}{2}$.  
        \end{definition}

        With this in hand, we can now define a sticky coupling of the Glauber dynamics for the DGFF, using $\mathsf{P}$ as its building block. 
                	
        \begin{definition}\label{def:DGFF-sticky-coupling}
        The sticky coupling $\mathsf{P}$ between two DGFF Glauber dynamics chains $\bh^{\bg}$ and $\bh^{\bg'}$ is defined by using the same Poisson noise for the update sequence, and whenever the height at a site $x\in \Lambda_n$ is being updated, using the sticky coupling of Definition~\ref{def:Gaussian-sticky-coupling} on the two Gaussian updates corresponding to the two dynamics, at $x$. 
        \end{definition}
        
        By monotonicity of $\mathsf{P}_\mu$, the sticky coupling of the DGFF chains is monotone i.e., if $\bg \le\bg'$ pointwise, then $\bh^\bg(t) \le \bh^{\bg'}(t)$ pointwise for all $t$. In particular, under the sticky coupling of $\bh^\bn$ and $\bh^\bg$, we have that $dh_x(t)$ and $V_t$ are non-negative for all $t$. 

        At this point we make a simple but useful observation. Unlike the identity coupling, the sticky coupling lends itself to discrepancies $dh_x(t)$ that are either zero or order-one-sized. We will use the following simple calculus exercise to prove this.   
        \begin{lemma}\label{lem:domination-of-bernoullis}
        Let $\nu_1^\mu$ be as in~\eqref{eq:nu-2-3}. There exists a universal constant $0<\zeta \le 1$ such that if $Y$ is drawn as $\bar \nu_1^\mu$ ($\nu_1^\mu$ normalized to be a probability measure), then 
        \begin{align*}
            Y\succeq \zeta \cdot \mbox{Ber}(\zeta) \qquad \mbox{for all $\mu>0$}\,,
        \end{align*}
        where $\succeq$ denotes stochastic domination.
        \end{lemma}
        \begin{proof}
                    Evidently, $\bar \nu_1^\mu$ is supported on $\mathbb R_+$.
                    It therefore suffices for us to show that there exist $\eta,\delta>0$ such that 
                    \begin{align}\label{eq:nu-1-density-lower-bound}
                       \inf_{x\in [\mu+\eta, \mu+2\eta]} \frac{d\bar\nu_1^\mu(x)}{dx} \ge \delta \qquad \mbox{uniformly over all $\mu>0$}\,.
                    \end{align}
                    This is sufficient via the choice of $\zeta = \min\{\eta,\delta \eta\}$ by non-negativity of $\mu$. 
                    
            We first consider $\mu$ small. Consider the scaling of the normalization factor $\nu_1^\mu(\mathbb R_+)$ in $\mu$. For a standard normal random variable $Z$, we have  
            \begin{align*}
                \nu_1^\mu (\mathbb R_+) = \mathbb P(Z\ge -\mu) - \mathbb P (Z\ge \mu) = \erf(\frac{\mu}{\sqrt 2}) = \frac{\sqrt{2}}{\sqrt{\pi}}\mu(1+o_\mu(1))\,
            \end{align*}
            where the $o_\mu(1)$ term goes to zero as $\mu \downarrow 0$. 
            By differentiating, we simultaneously have that, as long as $x\in K$ for some compact set $K$,  
            \begin{align*}
            \frac{1}{\mu}(e^{ - (x-\mu)^2/2}- e^{ - (x+\mu)^2/2}) = xe^{-x^2/2}(1+o_\mu(1)) 
            \end{align*}
            where the $o_\mu(1)$ term is uniform over all $x\in K$. 
            Therefore, 
                \begin{align*}
                   \inf_{x\in [\mu+\eta, \mu+2\eta]} d\bar \nu_1^\mu(x) = \frac 12(1+o_\mu(1)) xe^{- x^2/2} \ge \frac 12(1+o_\mu(1)) \eta e^{ -2\eta^2}\,.
                \end{align*}
            As a consequence, for any $\eta$, there exists a choice of $\delta$ such that uniformly over all $\mu<\mu_0$ for some $\mu_0$, the right-hand side is at least that $\delta$. 
 
            Let us now consider $\mu \ge \mu_0$. For all such $\mu$, we have for all $x\in [\mu+\eta,\mu+2\eta]$, 
        \begin{align*}
            d\bar \nu_1^\mu(x) \ge d\nu_1^\mu(x) &  = (2\pi)^{-1/2} (e^{-2\eta^2} - e^{ -2\mu^2}) 
        \end{align*}
        at which point as long as $\eta<\mu_0/2$, there will be a uniform choice of $\delta$ such that this is at least $\delta$ uniformly over $\mu \ge \mu_0$. Together, these imply the claimed~\eqref{eq:nu-1-density-lower-bound}. 
\end{proof}

    \subsection{Regularity estimates}
    Before using the above property of the sticky coupling to contract the volume supermartingale in an $O(n^2 \log \log n)$ time, we will need some a priori regularity estimates on the various processes $\bh^\bn,\bh^\bg$ and $d\bh$ that will hold for all $t \ge t_\star$.  
    
    We begin by proving the claimed upper bound on $V_t$ after the identity coupling phase. 
    
        \begin{proof}[\textbf{\emph{Proof of Lemma~\ref{lem:stage-1}}}]
        For any two initializations $\bh_1(0)$ and $\bh_2(0)$, if the corresponding DGFF dynamics $\bh_i^\cP(t)$ are coupled via the identity coupling, then by \eqref{eq:h-P}, 
        		\begin{align*}
		\bh_1^{\cP}(t) - \mathbb E[\bh_1^{\cP}(t)] = \bh_2^{\cP}(t) - \mathbb E[\bh_2^{\cP}(t)]\,.
		\end{align*}
        Thus, conditional on any Poisson update sequence $\cP$, under the identity coupling, we have  
        \begin{align*}
         dh_x^\cP(t) = h_x^{\bn,\cP}(t)  - h_x^{\bg,\cP}(t) = \mathbb E[h_x^{\bn,\cP}(t)]  - \mathbb E[ h_x^{\bg,\cP}(t)]\,, \qquad \mbox{for all $x\in \Lambda_n$}\,.
        \end{align*}
        This expectation is then upper bounded as  
        \begin{align*}
            \mathbb E[h_x^{\bn,\cP}(t)] - \mathbb E[h_x^{\bg,\cP}(t)]\le n\mathbb P(S_{x,-t}^{\cP}\notin \partial \Lambda_n) - \mathbb E[g(S_{x,-t}^{\cP})] \le 2n \mathbb P(S_{x,-t}^{\cP} \notin \partial \Lambda_n)
        \end{align*}
By Proposition~\ref{prop:mean-of-h}, with probability  $1-o(n^{-4})$,  $\cP_\infty$ is such that the conclusion of the proposition holds at $t  = t_\star +s n^2$. For any such $\cP_\infty$,  
        applying the bound of Proposition~\ref{prop:mean-of-h}, we see that for every $x\in \Lambda_n$, we have 
        \begin{align*}
            dh_x^{\cP}(t) \le 2n(\widehat \varphi_\one (x) + o(1)) e^{ - \lambda_\one t}\,.
        \end{align*}
        By definition of $t_\star$ and the choice of $t$, this is at most $2(\widehat \varphi_\one(x) + o(1))e^{-\pi^2 s/2}$. 
    \end{proof}
    
    It will also be helpful for us to use this mean bound, together, importantly, with the Gaussianity of $\bh^\bn(t)$ and $\bh^\bg(t)$, to get logarithmic maximal bounds on these processes and hence their discrepancies.

    \begin{lemma}\label{lem:max-bound}
    For every $\bg$, such that $\|\bg\|_\infty \le n$, we have 
    \begin{align*}
        \mathbb P\big(\|\bh^\bg(t)\|_\infty \le C\log n \mbox{ for all $t_\star \le t \le n^3$}\big) = 1-o(1)\,.
    \end{align*}
    As a result, under \emph{any} coupling, $\mathbb P(\|d\bh(t)\|_\infty \le 2C \log n \mbox{ for all $t_\star \le t\le n^3$})= 1- o(1)$. 
    \end{lemma}
    \begin{proof}
        We use the fact that we know the process is Gaussian, together with the fact that its variance \emph{increases} to that of the DGFF, to bound its maximum. We will union bound over all times between $t_\star$ and $n^3$, and all sites $x\in \Lambda_n$ to obtain the desired result. In order to union bound over the set of times, condition on the Poisson update sequence $\cP$ (in between the updates, the processes $\bh^\bg(t)$ are unchanged).  
        
        Fix any update time $t\in \cP$ such that $t_\star \le t\le n^3$. For any fixed $x\in \Lambda_n$, evidently $h^\bg(x,t)$ is a Gaussian with mean satisfying
        $$|\mathbb E[g(S_{x,-t}^{\cP})]|\le n\mathbb P(S_{x,-t}^{\cP}\notin \partial \Lambda_n) \le n (\widehat\varphi_\one (x)+o(1)) e^{ - \lambda_\one t}\,,
        $$
        and per Proposition~\ref{prop:mean-and-variance-of-h}, and Fact~\ref{fact:green-function-srw}, 
        $$
        \var(h^{\bg,\cP}(x,t)) \le G(x,x)  = O(\log n)\,.
        $$ 
        Since $t \ge t_\star$, with probability $1-o(n^{-8})$, $\cP$ is such that the mean is $O(1)$, say at most some universal constant $C_1$, and the variance is at most $C_2 \log n$ for some other universal constant. We see that this implies a uniform upper bound on the tails of $|h^\bg_x(t)|$, yielding that with probability $1-o(n^{-8})$, $\cP$ is such that for large enough $C$, 
        \begin{align*}
            \mathbb P ( |h^{\bg,\cP}(x,t)|\ge C\log n) = o(n^{ - 8})\,.
        \end{align*}
        Summing over the $n^2$ many vertices $x\in \Lambda_n$, for any fixed update time, we have that with probability $1-o(n^{-6})$, for large enough $C$, 
        $$
        \mathbb P\Big(\|\bh^{\bg,\cP}(t)\|_\infty \ge C\log n\Big) =o(n^{-6})\,.
        $$
        Now, with probability $1-o(1)$, there are at most $2n^3$ many Poisson clock updates between $t_\star$ and $n^3$, and we can union bound the above probabilities over these $n^3$ many Poisson clock ring times to deduce the desired bound. 
    \end{proof}

    \subsection{Bringing the volume supermartingale down}
We now move on to the analysis of the second phase in which we have switched to the sticky coupling. This analysis involves multiple stages where the fact that $V_t$ continues oscillating enough to eventually hit zero is established using slightly different types of arguments depending on how close to zero the process is. The remainder of this section is dedicated to the most substantive of these stages, which is getting the volume supermartingale down from $n^2(\log n)^{-5}$ (which upper bounds its value after time 
\begin{align*}
    T_0 = t_\star + \frac{12}{\pi^2} n^2 \log \log n\,,
\end{align*}
per Lemma~\ref{lem:V-t-exponential-decay}) to poly-logarithmic $(\log n)^5$ scale, and thus establishing Proposition~\ref{prop:V-t-down-to-polylog-scales}. 
    
    In order to perform this, we localize $V_t$ to a cascade of scales $[L_{i+1},L_{i-1}]$ given by 
    \begin{align*}
       L_i = n^{2}(\log n)^{-5-i}
    \end{align*}
   for $i\ge 0$. Then for $i\ge 0$, consider the sequence of hitting times
    \begin{align}\label{eq:cT-i}
        \cT_i = \inf \{t \ge T_0: V_t \le L_i\} \qquad \mbox{and let} \qquad I = \inf\{i: L_i \le (\log n)^5\}\,.
    \end{align}
    In particular, $I = O((\log n)/(\log \log n))$. The goal of this stage of the proof, as indicated, is to show that $\cT_I  - \cT_0 = O(n^2 \log \log n)$.

    The approach we take to proving this is to leverage the fact that $V_t$ is a non-negative jump super-martingale (see Claim~\ref{clm:V-t-supermartingale} below), so if we can lower bound its angle bracket $\langle V\rangle_t$, defined as follows, we can upper bound its hitting time of lower scales. 
    
    \begin{definition}\label{def:supermartingale-angle-bracket}
    By Doob's decomposition theorem, an adapted process $(V_t)_{t\ge 0}$ has a predictable part, and a martingale part, say $(M_t)_{t\ge 0}$. We denote by the angle bracket $\langle V \rangle_t$, the predictable quadratic variation of $M_t$, which is to say the unique predictable process such that $(M_t^2  - \langle V \rangle_t)_{t \ge 0}$ is a martingale. 
    \end{definition}
    
    In what follows we will abuse the terminology \emph{quadratic variation} to mean this angle bracket. 
    For technical reasons, it will be preferable to work with a supermartingale with bounded, rather than unbounded increments. Towards that, let
    \begin{align}\label{eq:cR-i}
        \cR_i = \inf\{t\ge \cT_i : V_t \ge L_{i-1}\}\,, \qquad \mbox{and} \qquad \cR= \min_i \cR_i\,.
    \end{align}
    and introduce the following truncated version of $V_t$, 
    \begin{align}\label{eq:tilde-A}
        \widetilde V_t = \begin{cases}
            V_t & t\le \cR \\
            L_i & t\ge \cR \mbox{ and } \cR = \cR_i <\cR_{i+1}
        \end{cases}
    \end{align}
We begin by establishing that $V_t$ and in turn $\widetilde V_t$ are supermartingales. 

\begin{claim}\label{clm:V-t-supermartingale}
The processes $(V_t)_{t\ge 0}$ and $(\widetilde V_t)_{t\ge 0}$ are jump supermartingales. Furthermore, for $t \in [\cT_i, \cT_{i+1})$, the jumps of $\widetilde V_t$ are bounded in absolute value by $L_{i-1}$.  
\end{claim}

\begin{proof}
To see the fact that $V_t$ is a supermartingale, it suffices to consider its expected change between $t$ and $t'$ where $t'$ is the first time after $t$, when an update is made at some $x\in \Lambda_n$. Given $\cF_t$ (the filtration given by the history of the chains $\bh^\bn,\bh^\bg$ up to time $t$), the site at which this update occurs is uniformly distributed in $\Lambda_n$. If the update occurs at site $x$, then $\mathbb E[dh_x(t')]$ will be given by 
$$\mathbb E[h^\bn(x,t')] - \mathbb E[h^\bg(x,t')] = \Delta h^\bn(x,t) - \Delta h^\bg(x,t)\,.$$
As such, we have 
\begin{align*}
\mathbb E[V_{t'} - V_t \mid \cF_t] & = \frac{1}{|\Lambda_n|} \sum_x \Big(\frac{1}{4} \sum_{y\sim x} dh_y(t)   - dh_x(t)\Big)\,.
\end{align*}
This double sum has $y\in \Lambda_n$ appearing four times if $y\not \sim \partial \Lambda_n$, but only two or three times if $y\sim \partial \Lambda_n$.   Under any monotone coupling, $dh_y(t)$ is non-negative, and therefore 
\begin{align*}
	\mathbb E[V_{t'} - V_t \mid \cF_t] \le \frac{1}{|\Lambda_n|} \Big(\sum_{y\in \Lambda_n} dh_y(t) -   \sum_{x\in \Lambda_n} dh_x(t)\Big) = 0\,.
\end{align*}
This implies that $V_t$ is a jump supermartingale as claimed. To see that $\widetilde V_t$ is also a jump supermartingale, note by definition that the jumps of $\widetilde V_t$ are deterministically less than or equal to the jumps of $V_t$. The bound on the jumps of $\widetilde V_t$ follow from the fact that it is non-negative, and its definition. 
\end{proof}

We now use the supermartingale property of $V_t$ to establish that the effect of truncation in~\eqref{eq:tilde-A} is typically not felt; namely with high probability $V_t = \widetilde V_t$ for all time.

        \begin{lemma}\label{lem:V-t-localization}We have   
    \begin{align*}
        \mathbb P\Big( V_t \le \frac{\log n}{(\log\log n)^{1/2}} L_i \quad \mbox{for all $t\ge \cT_i$, for all $i\le I$}\Big) = 1-o(1)\,. 
    \end{align*}
    In particular, $\mathbb P( V_t = \widetilde V_t \mbox{ for all $t\ge 0$}) = 1-o(1)$. 
    \end{lemma}
    
    \begin{proof}
By application of a standard maximal inequality for non-negative super-martingales, known sometimes as Ville's maximal inequality (see e.g.,~\cite[Eq.~(5.6)]{CLL-nabla-phi}), we have for each $i$, 
$$\P\Big(\sup_{t\ge \cT_i} V_t \ge r \Big)\le r^{-1} \mathbb E[V_{\cT_i} ] \le r^{-1} L_i\,,$$
for every $r$. In particular, for every $i$, 
$$\P\Big(\sup_{t\ge \cT_i}V_t \ge \frac{\log n}{(\log \log n)^{1/2}} L_i  \Big)\le \frac{(\log \log n)^{1/2}}{\log n}\,.$$
A union bound over all $i\le I$ (note that $I = o(\log n/(\log \log n)^{1/2})$, yields the desired.
    \end{proof}

    We next show that if the supermartingale accumulates a certain amount of quadratic variation when initialized in the interval $[L_{i+1}, L_{i}]$, then it must have left that interval, and with high probability it will exit downwards (i.e., at $L_{i+1}$). This follows from the following upper bound on the quadratic variation between the hitting times $\cT_i$ and $\cT_{i+1}$.

    \begin{lemma}\label{lem:angle-bracket-upper-bound}
    Let $\widetilde V_t$ be as in~\eqref{eq:tilde-A}, and let $\langle \widetilde V\rangle_t$ be as in Definition~\ref{def:supermartingale-angle-bracket}: we have, 
    \begin{align*}
        \mathbb P\Big(\langle \widetilde V\rangle_{\cT_{i+1}} - \langle \widetilde V\rangle_{\cT_i}  < 4 L_{i-1}^2 \mbox{ for all $0\le i \le I$}\Big)= 1- o(1)\,.
    \end{align*}
    \end{lemma}
    \begin{proof}
        For every fixed $i$, on the event of $\cT_{i}<\infty$, conditionally on $\cF_{\cT_i}$, we can consider the process $(\bar V_s^{(i)})_{s\ge 0}$ where 
        \begin{align*}
            \bar V_s^{(i)} = \widetilde V_{\cT_{i} + s}\,.
        \end{align*}
        By Claim~\ref{clm:V-t-supermartingale}, this is a supermartingale, and its increments are deterministically bounded by $L_{i-1} = n^{2} (\log n)^{-5 - (i-1)}$. 
        As such, by a diffusivity estimate on supermartingales (see~\cite[Proposition 21]{CLL-mixing-simplex} as well as~\cite{LL-WASEP}), for all $v\ge 4L_{i-1}^2$, 
        \begin{align*}
            \mathbb P(\langle \bar V^{(i)}\rangle_{\cT_{i+1} - \cT_{i}} \ge v ) \le 8(L_{i} - L_{i+1})v^{-1/2} \le 8n^{2}(\log n)^{-5 -i}v^{-1/2}\,.
        \end{align*}
        Plugging in $v = 4n^{4}(\log n)^{-10-2(i-1)}$, we see that the above is at most $O((\log n)^{-1})$ at which point a union bound over $i \le I$ completes the proof since $I = o(\log n)$.  
    \end{proof}

    The following counterpart lower bound on the accumulated quadratic variation between times $\cT_i$ and $\cT_{i+1}$ will be our main technical contribution in the remainder of this section; when combined with the above lemma, we get an upper bound of $2^{-i} n^2$ on $\cT_{i+1} - \cT_{i}$. 
    
    \begin{proposition}\label{prop:angle-bracket-lower-bound}
    There exists $\epsilon>0$ such that the following holds: 
    \begin{align*}
        \mathbb P\Big(\partial_t \langle \widetilde V\rangle_t \ge \frac{1}{(\log n)^2} V_t\,\, \mbox{ for all $\cT_0 \vee (T_0 + (\log n)^2) \le t\le \cT_I$}\Big) = 1-o(1)\,.
    \end{align*}
    \end{proposition}
    (At first glance, the lower bound on $t$ by $T_0 + (\log n)^2$ may seem bizarre---indeed it is technical and only there to ensure that every site has been updated at least once under the sticky coupling phase of Definition~\ref{def:two-stage-coupling}.)
    
    \begin{remark}\label{rem:quadratic-variation-bound-difficulty}
    In~\cite{CLL-nabla-phi}, a similar overall strategy was used in dimension $d=1$ to get cutoff for general $\nabla \phi$ interfaces, including the natural 1D counterpart of the DGFF. In the lower bound used in~\cite{CLL-nabla-phi}, the time-derivative of the angle bracket is bounded by the minimum of $V_t/\|d\bh(t)\|_\infty$ and $V_t^2/|\Lambda_n|$, corresponding to cases where either most of the contribution to $V_t$ is from discrepancies that are $\Omega(1)$, or from discrepancies that are $o(1)$ respectively.   
    In our 2D setting, the bound translates to $$\partial _t \langle \widetilde V\rangle_t \ge \min\Big\{\frac{V_t}{(\log n)^2},\frac{V_t^2}{n^2}\Big\}\,.$$ If the second quantity in the minimum is the smaller one, then, the upper bound we would get on $\cT_{i+1}-\cT_i$ would itself be of order $n^2 (\log n)^2$, and a similar issue would persist no matter our choices of scales $(L_i)_i$. As such, we must show that over time, the discrepancy process always maintains a property that the dominant contribution to $V_t$ is discrepancies that are $\Omega(1)$, or in other words that $V_t$ is comparable to the number of non-zero discrepancies. In that case we could use the first term in the minimum above as the lower bound on the angle bracket's time derivative. This is where the special ``zero-one" nature of the sticky coupling will be leveraged.
    \end{remark}
    
    \begin{remark}\label{rem:higher-d-sticky-coupling}
    Let us consider the same calculation for the sticky coupling for the DGFF on $\mathbb Z^d$ in $d\ge 3$. We consider two stages of analysis as above, first using the exponential decay of the expectation of $V_t$ to a time $t_\star(d)$ to be optimized, then leveraging the quadratic variation of $V_t$ to coalesce in a further $o(n^2 \log n)$ time. In the second stage of this, disregarding polylogarithmic factors, by~\eqref{eq:angle-bracket-growth} the quadratic variation accumulated by $\widetilde V_t$ would be at most $(t-t_\star) V_{t_\star}$, while the quadratic variation it needs to accumulate for there to be a decent chance of coalescence by time $t$ would now be at least $(V_{t_\star})^2$. In order for $t-t_\star = \widetilde O(n^2)$ and this strategy to work, $t_\star$ would need to be such that the expected value of $V_{t_\star}\le \widetilde O(n^2)$, i.e., $t_\star(d) \ge \frac{2d-2}{\pi^2} n^2 \log n$. This should be compared to the expected cutoff location of $T_d=\frac{d}{\pi^2}n^2 \log n$, corresponding to $V_{T_d}$ being of order $n^{1+\frac{d}{2}}$: that $T_d$ is the expected location, is seen by generalizing the lower bound argument to $d\ge 3$.
    \end{remark}
    
    Before getting into further details of the proof of Proposition~\ref{prop:angle-bracket-lower-bound}, let us conclude the proof of Proposition~\ref{prop:V-t-down-to-polylog-scales} by combining Lemma~\ref{lem:angle-bracket-upper-bound} with Proposition~\ref{prop:angle-bracket-lower-bound}. 
    
    \begin{proof}[\textbf\emph{Proof of Proposition~\ref{prop:V-t-down-to-polylog-scales}}]
        First of all, by Lemma~\ref{lem:stage-1}, if we recall the value $T_0  = t_\star + \frac{12}{\pi^2} n^2 \log \log n$, we have $\mathbb P(\cT_0 \le T_0) = 1-o(1)$. 
        Working on this event, our main claim is that the pair Lemma~\ref{lem:angle-bracket-upper-bound} and Proposition~\ref{prop:angle-bracket-lower-bound} imply that
     \begin{align}\label{eq:T-i-minus-T-i-minus-1}
            \mathbb P \Big(\cT_{i+1}- \cT_i \le 2^{-i} n^2 \mbox{ for all $i\le I$}\Big) = 1-o(1)\,.
        \end{align}
        This will of course imply the desired by summing out over $i$, to find that $\cT_I - \cT_{0} = O(n^2)$. In order to prove the claimed bound of~\eqref{eq:T-i-minus-T-i-minus-1}, notice that while $\cT_i \le t \le \cT_{i+1}$, we have $V_t \ge L_{i+1}$, so that by Proposition~\ref{prop:angle-bracket-lower-bound}, 
        \begin{align*}
            \langle \widetilde V\rangle_{\cT_{i}+s} - \langle \widetilde V\rangle_{\cT_{i}} \ge \frac{ (s - (\log n)^2) L_{i+1}}{(\log n)^2} = (s-(\log n)^2)  n^2 (\log n)^{-8-i}\,.
        \end{align*}
        At the same time, by Lemma~\ref{lem:angle-bracket-upper-bound}, with probability $1-o(1)$, for all $0\le i \le I$,
        \begin{align*}
            \langle \widetilde V\rangle_{\cT_{i+1}} - \langle \widetilde V\rangle_{\cT_i}\le 4L_{i-1}^2= 4 n^4 (\log n)^{-8-2i}\,.
        \end{align*}
        Now taking $s \ge 4 n^2 (\log n)^{-i} + (\log n)^2$, we see that the former lower bound would be larger than the latter upper bound, implying in particular that $\cT_{i+1}  - \cT_i$ must be at most $4 n^2 (\log n)^{-i} + (\log n)^2$. In particular, we obtain~\eqref{eq:T-i-minus-T-i-minus-1} (with some room to spare). Summing that up over $0\le i\le I$, we get 
        \begin{align*}
            \mathbb P \Big(\cT_I - \cT_0 \le 8 n^2\Big) = 1-o(1)\,.
        \end{align*}
        We then get the requisite bound of Proposition~\ref{prop:V-t-down-to-polylog-scales} by means of the definition of $\cT_I$. 
    \end{proof}
    
    \subsection{Establishing Proposition~\ref{prop:angle-bracket-lower-bound}}
   It remains for us to establish the lower bound on the angle bracket of $\widetilde V_t$ which is the goal of this section. As already indicated, this will be done using crucially the ``zero-one" nature of discrepancies under the sticky coupling, which we formalize next. 
    In what follows, let 
    \begin{align}\label{eq:def-At-Nt}
        \cA_t = \{x\in \Lambda_n: dh_x(t) \ne 0\}\,, \qquad \mbox{and let} \qquad \cN_t = |\cA_t| = \sum_{x\in \Lambda_n} \mathbf 1\{dh_x(t)\ne 0\}\,.
    \end{align}
    
    The key estimate we rely on is the following lemma capturing the fact that for each $x\in \Lambda_n$, on the event of disagreement, the discrepancy $dh_x(t)$ is typically an order one random variable. It will apply for all times after $T_0+ (\log n)^2$, at which time by a coupon collecting argument, all sites have been updated at least once under the sticky coupling. 
    
    \begin{lemma}\label{lem:two-norm-to-0-norm}
    There exists a universal constant $c>0$ such that the following holds. Fix any time $t\ge T_0 + (\log n)^2$. Suppose $\cP$ is such that every site is updated in the window $[t_0, t]$ for $t_0 = t- (\log n)^2$. Then for all $r$, 
    \begin{align*}
        \mathbb P \Big(\sum_{x\in \Lambda_n} [(dh_x(t))^2 \wedge 1] \le c \cN_t\, ,\, \cN_t \ge r \given \cF_{t_0}, \cP\Big) = \exp ( - r/C)\,.
    \end{align*}
    \end{lemma}
    \begin{proof}
        Condition on $\cF_{t_0}$, as well as the entire update sequence (but not the Gaussian random variables) in the time interval $[t_0, t]$, and suppose that $\cP([t_0,t])$ is such that every site in $\Lambda_n$ is updated at least once. 
        
        We claim that conditionally on this, there exists a coupling between $(dh_x(t))_x$ and an i.i.d.\ sequence $(\chi_i)_{i} \sim \zeta\cdot \ber(\zeta)$ for $\zeta$ given by Lemma~\ref{lem:domination-of-bernoullis} such that the following holds. Recall $\cA_t$ from~\eqref{eq:def-At-Nt}, and let $(y_i)_i$ be the vertices of $\cA_t$ enumerated according to the times $t_i \in [t_0, t]$ of their last (i.e., most recent) update so that $t_i<t_{i+1}$ for every $i$; then with probability one
        \begin{align}\label{eq:claimed-coupling-property}
            dh_{y_i}(t) \ge \chi_i \qquad \mbox{for all $1\le i\le \cN_t$}\,.
        \end{align}
        Let us conclude the proof given such a coupling, before proving the existence of the claimed coupling. Given \eqref{eq:claimed-coupling-property}, we in particular have that 
        \begin{align*}
            \sum_{x\in \Lambda_n} [(dh_{x}(t))^2 \wedge 1]  = \sum_{y_i \in \cA_t} [(dh_{y_i}(t))^2 \wedge 1] \ge \sum_{1\le i\le \cN_t} \chi_i^2\,,
        \end{align*}
        where the truncation by $1$ didn't affect anything because $\chi_i \le 1$ deterministically.
        Therefore, under this coupling, we have 
        \begin{align*}
            \mathbb P \Big(\sum_{x\in \Lambda_n} [(dh_x(t))^2\wedge 1] \le c \cN_t\, ,\, \cN_t \ge r \mid \cF_{t_0}, \cP\Big) &  \le \mathbb P\Big(\sum_{i\le \cN_t} \chi_i^2\le c \cN_t\,,\, \cN_t \ge r\mid \cF_{t_0}, \cP\Big) \\
            & \le \mathbb P\Big(\sum_{i\le N} \chi_i^2 \le cN \mbox{ for some $r \le N \le |\Lambda_n|$}\Big)\,.
        \end{align*}
        This inequality then concludes the proof by application of the following  concentration fact, whose proof is standard and follows from Hoeffding's inequality together with a union bound. 
            \begin{fact}
    Suppose $(\chi_i)_{i\ge 1}$ are a sequence of $\zeta \cdot \mbox{Ber}(\zeta)$ random variables, for $\zeta>0$. Then, for every $r\ge 0$, 
    \begin{align*}
        \mathbb P\Big(\sum_{1\le i\le N} \chi_i \le  \frac{\zeta^2 N}{2} \mbox{ for some $r \le N \le n^{2}$}\Big) \ge 1-e^{ - r/C}\,.
    \end{align*}
    \end{fact}
        
        It remains to construct the claimed coupling attaining~\eqref{eq:claimed-coupling-property}. Per the filtration $\cF_{t_0}$, generate the pair of configurations $(\bh^\bn,\bh^\bg)$ at time $t_0$. Reveal the Poisson update sequence on the interval $[t_0, t]$, and let $(x_i)_i$ be an enumeration of all vertices of $\Lambda_n$ according to the times $s_i \in [t_0,t]$ of their last update. We construct the claimed coupling iteratively, reasoning that for all $i$, either $dh_{x_i}(s_i) = 0$ or otherwise $x_i = y_\ell$ for
        $$\ell = \ell(i) = |\{j<i: dh_{x_j}(s_j)\ne 0\}| + 1\,,\qquad \mbox{and}\qquad dh_{y_\ell}(s_i) \ge \chi_\ell\,.$$ 
        
        Establishing this is sufficient because the vertex $x_i$ is not updated again between time $s_i$ and time $t$. Now fix any $i$ and condition on the filtration $\cF_{s_{i}^-}$; the random variable $dh_{x_{i}}(s_{i})$ is then drawn as follows. First, draw a 
        $$
        \ber\Big(\erf\Big(\frac{\mu_{i}}{2\sqrt{2}}\Big)\Big) \qquad \mbox{where} \qquad \mu_{i} = \frac{1}{4}\sum_{z\sim x_i}dh_z(s_{i}^-)
        $$
        which dictates whether or not $dh_{x_i}(s_i)=0$ because the $\erf(\mu_{i}/(2\sqrt{2}))$ quantity above is exactly the total-variation distance between the Gaussians being drawn at $x_i$. If this Bernoulli is zero, then of course the claim holds; else, conditionally on the fact that the coupling of $\bh^\bn$ and $\bh^\bg$ at $x_i$ failed, then necessarily $x_i = y_{\ell(i)}$ and we set $dh_{y_\ell}(s_i)=2Y_{\ell}$ for $Y_{\ell}\sim \bar\nu_1^{\mu_{i}}$ defined per Definition~\ref{def:Gaussian-sticky-coupling}.

         By Lemma~\ref{lem:domination-of-bernoullis}, uniformly over $\mu_{i}$, the random variable $Y_\ell$ stochastically dominates $\zeta \ber(\zeta)$, so there is a coupling, which we use, of $(Y_{\ell},\chi_{\ell})$ such that with probability one, $Y_{\ell}\ge \chi_{\ell}$, concluding the proof.
    \end{proof}
    
    Lemma~\ref{lem:two-norm-to-0-norm} gives us a lower bound on the $\ell^2$-norm of the discrepancy process $d\bh(t)$ while $t\le \cT_I$. It will turn out that the time derivative of the angle bracket of $\widetilde V_t$ is comparable to the $\ell^2$-norm of the discrepancy process, so in order to get Proposition~\ref{prop:angle-bracket-lower-bound}, we need to improve Lemma~\ref{lem:two-norm-to-0-norm} to a comparison between the $\ell^2$-norm of the discrepancies, and the $\ell^1$ norm, namely $V_t$ itself. This is achieved in the following lemma.  
    
    \begin{lemma}\label{lem:two-norm-to-one-norm}
    There exists $C>0$ such that, if we define 
                    \begin{align*}
            \cD = \Big\{\sum_{x\in \Lambda_n} [(dh_x(t))^2 \wedge 1] \ge \frac{1}{C\log n} \sum_{x\in \Lambda_n} dh_x(t) \,\,\,\mbox{ for all $T_0 + (\log n)^2 \le t\le n^3 \wedge \cT_I$}\Big\}\,,
        \end{align*}
then we have $\mathbb P(\cD) = 1-o(1)$. 
    \end{lemma}
    \begin{proof}
        The inequality will proceed from comparing the right and left hand quantities in the definition of $\cD$ to the number of discrepancies, $\cN_t$. 
        
               On the one hand, we have deterministically, that 
        \begin{align*}
            \cN_t \ge \frac{1}{\max_x dh_x(t)} \sum_{x\in \Lambda_n} dh_x(t)
        \end{align*}
        On the event that $\|\bh^{\bn}(t)\| \vee \|\bh^{\bg}(t)\| \le C \log n$, which, for a sufficiently large $C>0$, holds for all $t_\star \le t\le n^3$ except with probability $o(1)$ by Lemma~\ref{lem:max-bound}, we evidently have $\max_{x\in \Lambda_n} dh_x(t) \le 2C\log n$. Thus, with probability $1-o(1)$,  
        \begin{align*}
            \cN_t  \ge \frac{1}{2C \log n} \sum_{x\in \Lambda_n} dh_x(t) \qquad \mbox{for all $t_\star \le t \le n^3$}\,.
        \end{align*}
        The proof is complete if we establish that for some $c>0$,  with probability $1-o(1)$, 
        \begin{align*}
            \sum_{x\in \Lambda_n} \big[(dh_x(t))^2 \wedge 1\big] \ge c \cN_t \qquad \mbox{for all } {T_0 + (\log n)^2 \le t\le n^3\wedge \cT_I\,}.
        \end{align*}
        To see this, notice first that for all $t< \cT_I$, we have $V_t \ge L_I$, so that  
                \begin{align*}
             \cN_t \ge \frac{1}{2C\log n} L_I\ge (\log n)^2 \qquad \mbox{for all $t_\star \le t\le n^3 \wedge \cT_I$}\,,
        \end{align*}
        except with probability that is $o(1)$. We work now on that event and condition on all the Poisson clock rings $\cP$ such that there are at most $2n^{2}n^{3} = 2n^5$ many total clock rings between time $T_0$ and $n^3$ (the complement of this event has probability $o(1)$ by standard concentration for Poisson random variables). On that event, we can union bound over the at most $2n^{5}$ many times at which the discrepancy process can change, at each time applying the bound of Lemma~\ref{lem:two-norm-to-0-norm} with the choice $r = (\log n)^2$. 
    \end{proof}
    
    With the above preparation, we are now ready to prove 
    Proposition~\ref{prop:angle-bracket-lower-bound}.
    \begin{proof}[\textbf{\emph{Proof of Proposition~\ref{prop:angle-bracket-lower-bound}}}]
        Let $\cC$ be the event of Lemma~\ref{lem:V-t-localization} and work on that event as it has probability $1-o(1)$. 
        For any $t$ and any fixed configuration at time $t_-$, let $q_{x,t}$ denote the total-variation distance at the site $x$ between the update distributions for the two height functions. In particular, $q_{x,t}$ is given by 
        \begin{align*}
            q_{x,t} = \erf\Big(\frac{\mu_{x,t}}{2\sqrt{2}}\Big) \qquad \mbox{where} \qquad \mu_{x,t} = \frac{1}{4}\sum_{y\sim x} dh_y(t_-)\,.
        \end{align*} 
        By definition of the angle bracket from Definition~\ref{def:supermartingale-angle-bracket}, if $t$ is such that $\cT_{i} \le t \le \cT_{i+1}\wedge \cR$, then the time derivative of the angle bracket of $\widetilde V$ is given precisely by
        \begin{align}\label{eq:angle-bracket-growth}
            \partial_t \langle \widetilde V\rangle_t  =  \sum_{x\in \Lambda_n} \left[(1-q_{x,t}) (dh_x(t_-))^2 + q_{x,t} \mathbb E[Z_{x,t}^2 \mid \cF_{t_-}]\right]\,,
        \end{align}
        for $Z_{x,t}$ being a random variable given by 
        \begin{align*}
            Z_{x,t} = (2Y_{x,t} - dh_{x}(t_-))\wedge (R -\widetilde V_{t_-})\,, \qquad \mbox{where}\qquad R= n^{2}(\log n)^{-5-(i-1)}\,,
        \end{align*}
        and $Y_{x,t}$ having density $\bar \nu_1^{\mu_{x,t}}$ per Definition~\ref{def:DGFF-sticky-coupling}. 
        
        In the above expression for $\partial_t \langle \widetilde V\rangle_t$, the first term in the sum takes care of the situation where the sticky coupling fully couples the two Gaussians, and the second term, the case where the sticky coupling does not succeed. Since we are working on the event $\cC$, we have that $n^2 (\log n)^{-5 - (i-1)}- \widetilde V_{t_-}$ is at least $R/2$, and for $i\le I$, we have that $R/2 \ge 1$. Then, by~\eqref{eq:nu-1-density-lower-bound}, there exists a universal $\epsilon>0$ (uniform over $\mu_{x,t}$) such that uniformly over $\cF_{t_-}$,  
        \begin{align*}
            \mathbb E[Z_{x,t}^2\mid \cF_{t_-}] \ge \epsilon\,.
        \end{align*}
        As a consequence, for some other universal $\epsilon>0$, we get  
        \begin{align*}
            \partial_t \langle \widetilde V\rangle_t \ge \epsilon\sum_{x\in \Lambda_n} [(dh_x(t_-))^2 \wedge 1]\,.
        \end{align*}
        At this point, an application of Lemma~\ref{lem:two-norm-to-one-norm} evidently concludes the proof. 
    \end{proof}

	\section{Proof of mixing time upper bound}
	In this section, we conclude the proof of the mixing time upper bound by combining the results from previous sections to couple the Glauber dynamics chains from initializations $\bn$ and an arbitrary initialization $\bg$ respectively. 
	
	Let $(dh_x(t))$ be the discrepancy process, and let $V_t$ be the volume supermartingale as in~\eqref{eq:V-t}. Recall that the main result of Section~\ref{sec:volume-supermartingale} was establishing that for $t = t_\star + O(n^2 \log \log n)$, we have $V_t = O((\log n)^C)$. The aim of this section is to show that in a further $O(n^2\log \log n)$ time under the sticky coupling, it hits zero and the two chains coalesce.  There are several ways to do this and though all are reasonably straightforward, our approach involves breaking the analysis into two stages, getting $V_t$ to be $o((\log n)^{-C}),$ and then, from there down to zero.  
	
	\subsection{Getting the volume martingale to be $o(1)$}
	Once the volume martingale has been reduced to a polylogarithmic scale, its natural exponential contraction can be used to show it gets microscopically close to zero in a further $O(n^2 \log \log n)$ time. A more refined version of the arguments of Section~\ref{sec:volume-supermartingale} could probably establish this in a shorter $O(n^2)$ amount of time, but we make no attempt to optimize this since we anyways paid an $O(n^2 \log \log n)$ amount of time beyond the cutoff time of $t_\star$ in Section~\ref{sec:volume-supermartingale}.

	Throughout this section, let $T_I = t_\star + C_I n^2 \log \log n$ for a constant $C_I$ sufficiently large such that by Proposition~\ref{prop:V-t-down-to-polylog-scales}, with high probability $\cT_I \le T_I$ and by Lemma~\ref{lem:V-t-localization}, $V_{T_I} \le (\log n)^{6}$ with probability $1-o(1)$. 
	\begin{lemma}\label{lem:V-t-exponential-decay}
	There exists $C_F>0$ such that 
	\begin{align*}
	   \mathbb P\Big( V_{T_I + C_F n^2 \log \log n} \ge (\log n)^{-4} \Big)  = o(1)\,.
	\end{align*} 
	\end{lemma}
	\begin{proof}
	Condition on $\cF_{T_I}$, and consider for $t > T_I$, 
	\begin{align*}
	    \mathbb E [ V_t \mid \cF_{T_I}] = \mathbb E_{\cP} \Big[\sum_{x\in \Lambda_n} \mathbb E[d \bh^\cP_x(t) \mid \cF_{T_I}]\Big] = \mathbb E_{\cP} \Big[\sum_{x\in \Lambda_n}\mathbb E[d\bh(S_{-(t-T_I)}^{x,\cP},T_I)]\Big]\,.
	\end{align*}
	This then fits into the framework of Lemma~\ref{lem:annealed-expectation-against-random-walk}, from which we deduce that as long as $t\ge T_I + n^2$, we have 
	\begin{align*}
	    \mathbb E[ V_t \mid \cF_{T_I}] \le Ce^{ - \lambda_\one (t- T_I - n^2)} V_{T_I}\,.
	\end{align*}
	Taking $t=  T_I + C_F n^2 \log \log n$ for a $C_F$ sufficiently large, and noting that with probability $1-o(1)$, $V_T \le (\log n)^6$, we see that with probability $1-o(1)$ we have 
	\begin{align*}
	    \mathbb E[V_t] \le (\log n)^{-6}\,,
	\end{align*}
	from which the claim follows by a Markov inequality. 
	\end{proof}
	
	\subsection{Coupling in one sweep from there}
	As a byproduct of Lemma~\ref{lem:V-t-exponential-decay}, we know that after some $t_* + O(n^2 \log \log n)$ time, $V_t$ is at most $(\log n)^{-4}$. From there, we claim that the coalescence happens in essentially one final sweep of updates. Let 
	\begin{align*}
		\cT_F = \{\inf t : V_t \le (\log n)^{-3}\}\,,
	\end{align*}
	so that by Lemma~\ref{lem:V-t-exponential-decay}, with probability $1-o(1)$, $\cT_F \le T_F$. 
	
	\begin{lemma}\label{lem:coalesce-final-sweep}
	For every $s \ge (\log n)^2$,   
	\begin{align*}
	    \mathbb P(V_{\cT_F + s} \ne 0 \mid \cF_{\cT_F}) = o(1)\,.
	\end{align*}
	\end{lemma}
	\begin{proof}
	    Condition on $\cF_{\cT_F}$ as well as the update sequence $\cP$. Let $\sigma$ be the smallest number such that every site's clock rings at some point in the interval $[\cT_F,\cT_F + \sigma]$. Suppose that $\cP$ is such that $\sigma \le s$, and no site is updated more than $(\log n)^2$ many times in the interval $[\cT_F, \cT_F + \sigma]$. The intersection of these events has probability $1-o(1)$ by standard coupon collecting arguments. 
	    
	    We will union bound over the probability of the sticky coupling failing to couple the two Gaussians in any one of the updates in the time interval $[\cT_F, \cT_F+ \sigma]$.  
	    In particular, we union bound over the probability that some update $(\cT_F+s_i)_{i: s_i \le \sigma}$ is the \emph{first} update to fail to couple. (If all updates in the interval $[\cT_F, \cT_F + \sigma]$ succeed at coupling, then we are done since all sites in $\Lambda_n$ are updated at least once during that period.) For any $s_i$, the probability that the update at $\cT_F + s_i$ at site $x_i$ is the first one to fail is at most the probability that the sticky coupling of the two Gaussians fails, maximized over all possible realizations of $d\bh(\cT_F+s_i^-)$ that are pointwise smaller than $d\bh(\cT_F)$. The fact that we only need to consider $d\bh(\cT_F+s_i^-)$ that are pointwise smaller than $d\bh(\cT_F)$ is because if all prior updates in $[\cT_F, \cT_F + s_i^-]$ have succeeded, then necessarily $d\bh(\cT_F+s_i^-)$ is pointwise smaller than $d\bh(\cT_F)$. This failure probability at time $\cT_{F}+s_i$ is then bounded by the total-variation distance between standard Gaussians whose mean differs by 
	    \begin{align*}
	        \frac 14 \sum_{y\sim x}dh_y(\cT_F+s_i^-) \le \frac 14 \sum_{y\sim x} dh_y(\cT_F)\,.
	    \end{align*}
	    This total-variation distance is in turn at most  
	    \begin{align*}
	    \erf \Big( \frac{1}{8\sqrt{2}} \sum_{y\sim x} dh_y(\cT_F)\Big) \le \frac{1}{4\sqrt{2}\pi} \sum_{y\sim x} dh_y(\cT_F)\,.
	    \end{align*}
	    Summing over all updates in the period $[\cT_F, \cT_F+ \sigma]$, we find that the union bound gives a failure probability bounded by 
	    \begin{align*}
	        C \sum_{x_i} \sum_{y\sim x_i} dh_y(\cT_F)\le 4C(\log n)^2 \sum_{x\in \Lambda_n} dh_x(\cT_F) = 4C (\log n)^2  V_{\cT_F}\,.
	    \end{align*}
	    This is evidently $o(1)$ by the definition of $\cT_F$. 
	\end{proof}
	
	\subsection{Upper bound on mixing time}
	We now stitch together the above stages in order to conclude the desired mixing time upper bound of Theorem~\ref{thm:main-revised}. 
	
	\begin{proof}[\textbf{\emph{Proof of  Theorem~\ref{thm:main-revised}: upper bound}}]
	In what follows, let $T = t_\star + C_\star n^2 \log \log n$ for a sufficiently large $C_\star$. It suffices for us to show that for every initialization $\bg$ satisfying $\|\bg\|_\infty \le n$, there exists a coupling of $(\bh^\bn, \bh^\bg)$ such that 
	\begin{align}\label{eq:wts-mixing-upper-bound}
	\|\mathbb P(\bh^\bn(T) \in \cdot) - \mathbb P (\bh^\bg(T) \in \cdot) \|_\tv = o(1)\,.
	\end{align}
	Indeed, given this, one concludes the mixing time bound as follows. First draw $\bg \sim \pi$, then observe that with probability $1-o(1)$,  $\bg \sim \pi$ satisfies $\|\bg\|_\infty \le n$, and use the coupling to $\bh^\bn$ for that initialization $\bg$; this yields a coupling attaining $\|\mathbb P(\bh^\bn(T) \in \cdot) - \pi\|_\tv = o(1)$. Then, uniformly over initializations $\bg$ having $\|\bg\|_\infty \le n$, by the above and a triangle inequality,
	\begin{align*}
	\|\mathbb P(\bh^\bg(T)\in \cdot) - \pi \|_\tv \le \|\mathbb P (\bh^\bg(T) \in \cdot ) - \mathbb P(\bh^\bn(T) \in \cdot)\| + \|\mathbb P(\bh^\bn(T) \in\cdot) - \pi\|_\tv = o(1)\,.
	\end{align*}
	We claim that the coupling described in Definition~\ref{def:two-stage-coupling}, 
	attains the inequality of~\eqref{eq:wts-mixing-upper-bound}. Indeed, since this coupling is monotone for all times, 
	\begin{align*}
		\|\mathbb P(\bh^\bn(T) \in \cdot) - \mathbb P (\bh^\bg(T) \in \cdot) \|_\tv \le \mathbb P (V_T \ne 0)\,.
	\end{align*}  
	If the constant $C_\star$ in the definition of $T$ is sufficiently large, this probability is $o(1)$ by combining Proposition~\ref{prop:V-t-down-to-polylog-scales}, Lemma~\ref{lem:V-t-exponential-decay} and Lemma~\ref{lem:coalesce-final-sweep}. 
	\end{proof}
	
	We now make a comment on Remark~\ref{rem:general-initializations} from the introduction. Suppose that we want to only consider the mixing time restricted to initializations $\bg$ such that $\|\bg\|_\infty \le a_n$ for a sequence $\frac{a_n}{\log n} \to \infty$. The proof of the mixing time upper bound would then be essentially identical to the proof of Theorem~\ref{thm:main-revised}, with the sole distinction being that the volume super-martingale is brought down to size $n^2 (\log n)^{ - 5}$ in time that is $t_\star(a_n) + O(n^2 \log \log n)$ in the resulting analogue of Lemma~\ref{lem:stage-1}. 

\section{Sharp total-variation distance lower bound}
In this section, we establish the lower bound on the total-variation distance to stationarity, and as a consequence, the lower bound for the mixing time attaining the expected cutoff profile. In the 1D case of the exlcusion process, in~\cite{Lacoin-cutoff-circle}, the cutoff profile is obtained by considering the convergence time of the first fourier mode of the height function, i.e., the inner product with $\varphi_\one$. By way of analogy, our lower bound is obtained by considering the test function 
\begin{align}\label{eq:F-test-function}
    F(\bh) = \langle \varphi_\one, \bh\rangle = \sum_{x\in \Lambda_n} \varphi_\one(x) h(x)\,.
\end{align}
Since constants will be relevant to the cutoff profile, we recall that $\varphi_\one$ is the normalized eigenvector so that $\langle \varphi_\one,\varphi_\one\rangle =1$, and we defined $\widehat \varphi_\one  = \varphi_\one \langle \varphi_\one , \one\rangle = \frac{8 n}{\pi^2}(1+o(1)) \varphi_\one$. We will prove Theorem~\ref{thm:main-lower-bound}, from which the corresponding lower bounds of Theorem~\ref{thm:main-revised} and its generalization in Remark~\ref{rem:general-initializations} follow immediately. 
	
	\begin{proof}[\textbf{\emph{Proof of Theorem~\ref{thm:main-lower-bound}: lower bound}}]
We construct initializations having $\|\bh(0)\|_\infty \le a_n$ attaining the claimed total variation lower bound. These initializations will be deterministic shifts of the stationary distribution, which will give us better control on their covariance profile over time than if they started from e.g., the maximal all $a_n$ initialization. Namely, given $\cP$, their time evolutions will simply be deterministic shifts of the stationary measure. This is formalized by the following lemma which follows from the form of the Glauber dynamics updates. 

\begin{lemma}\label{lem:stationary-with-shift}
Fix any $\vec{l} = (l_x)_{x\in \Lambda_n}$. Consider the DGFF Glauber dynamics initialized from (a random) $\bg$ where $\bg - \vec{l}\sim \pi$. Then for any fixed Poisson update sequence $\cP$, if we let 
$$f_x^\cP(t) = \mathbb E[\vec{l}(S_{-t}^{x,\cP})]\,,$$
then we have that 
\begin{align*}
(h_x^\cP(t)- f_x^\cP(t))\sim \pi\,.
\end{align*}
\end{lemma}

We defer the proof of the lemma to after the conclusion of the proof of the claimed cutoff profile lower bound. Let $C_0$ be a large enough constant that the maximum of a DGFF on $\Lambda_n$ is smaller than $C_0\log n$ with probability $1-o(1)$. 

Let $l_x = a_n - C_0 \log n$ for all $x\in \Lambda_n$ and consider the initialization $\bh(0) = \bg$ where $\bg - \vec{l} \sim \pi$. 
Notice that since $a_n/\log n\to \infty$, with probability $1-o(1)$, $\|\bg\|_\infty \le a_n$, so it suffices to prove the lower bound for this distribution over initializations. 

We also notice that by the definition of total-variation distance and a reverse triangle inequality, for any event $\Gamma$ for the Poisson update sequence, 
	\begin{align}\label{eq:tv-distance-lower-bound}
	\|\mathbb P(\bh(t) \in \cdot) - \pi\|_\tv & \ge \|\mathbb P ( F(\bh(t))\in \cdot) - \pi(F(\bh)\in \cdot)\|_\tv \nonumber\\
	& \ge \inf_{\cP\in \Gamma} \|\mathbb P(F(\bh^\cP(t))\in \cdot) - \pi(F(\bh)\in \cdot)\|_\tv \nonumber\\ 
	& \qquad - \sup_{\cP,\cP'\in \Gamma} \|\mathbb P(F(\bh^\cP(t))\in \cdot) - \mathbb P(F(\bh^{\cP'}(t)) \in \cdot) \|_\tv - \mathbb P(\cP\notin \Gamma) 
	\end{align}
	As such, it suffices for us to condition on the Poisson update times for some typical set $\Gamma$, and consider the two total variation distances above. 

Let $\Gamma$ denote the set of all $\cP$ such that Proposition~\ref{prop:mean-of-h} applies; notice that by that proposition then, $\mathbb P (\cP\notin \Gamma) = o(1)$. Let 
\begin{align*}
    T_s = t_\star(a_n) + \frac{\pi^2}{2} s n^2\,.
\end{align*}
For any such $\cP$, uniformly over $x\in \Lambda_n$, 
\begin{align*}
    \mathbb P \Big(S_{-T_s}^{x,\cP}\notin \partial \Lambda_n\Big) = \widehat \varphi_\one(x)(1+o(1))\frac{1}{a_n}e^{ - s}\,.
\end{align*}
Using the fact that $a_n$ goes to infinity faster than order $\log n$, we also have that $\vec{l} = a_n (1+o(1))$.  Therefore, uniformly over $\cP\in \Gamma$, 
\begin{align}\label{eq:mean-quantity-lb}
    \vec{f}^\cP(T_s) = \mathbb E[\vec{l}(S_{-T_s}^{x,\cP})]=  \frac{8n}{\pi^2} \varphi_\one (1+o(1))e^{-s}\,, \qquad \mbox{for all $x\in \Lambda_n$}\,.
\end{align}
We now need to bound two total-variation distances on the right-hand side of~\eqref{eq:tv-distance-lower-bound}.  
 For these total-variation computations, we use Lemma~\ref{lem:stationary-with-shift} to observe that
\begin{align*}
    \bh^\cP(t) \stackrel{d}= \cN(\vec{f}^\cP(t), G)\,.
\end{align*}
By linearity of $F$, $F(\bh^{\cP}(t))$ is Gaussian, and by an explicit calculation, it is distributed as 
\begin{align*}
    F(\bh^{\cP}(t)) \sim \cN\big(\langle  \varphi_\one, \vec{f}^\cP(t)\rangle, \langle \varphi_\one, G \varphi_\one\rangle \big) = \cN\big(\langle  \varphi_\one, \vec{f}^\cP(t)\rangle, \lambda_\one^{-1} \big)\,.
\end{align*}
Let us now consider the two total variation distances in~\eqref{eq:tv-distance-lower-bound}, beginning with the second which we wish to show is $o(1)$. Uniformly over $\cP, \cP'\in \Gamma$, factoring out a $\lambda_\one^{1/2}$, we have
\begin{align*}
    \|\cN(\lambda_\one^{1/2}\langle  \varphi_\one, \vec{f}^\cP(t)\rangle, 1) - \cN(\lambda_\one^{1/2}  & \langle  \varphi_\one, \vec{f}^{\cP'}(t)\rangle, 1)\|_\tv \\
    & =\erf\Big(\frac{\lambda_\one^{1/2}}{2\sqrt 2}\langle  \varphi_\one, \vec{f}^\cP(t) - \vec{f}^{\cP'}(t)\rangle\Big)\,.
\end{align*}
If we take $t = T_s$ and plug in $\lambda_\one = \frac{\pi^2}{2n^2}(1+o(1))$ and~\eqref{eq:mean-quantity-lb}, we find this to be $\erf(\frac{2}{\pi} e^{ - s} \cdot o(1))$ which is $o(1)$ by linearity of the $\erf$ function near zero. 

Now consider the first term in~\eqref{eq:tv-distance-lower-bound}. Similar to the above, 
\begin{align*}
    \|\mathbb P(F(\bh^\cP(T_s))\in \cdot) - \pi(F(\bh)\in \cdot)\|_\tv &  = \erf\Big(\frac{\lambda_\one^{1/2}}{2\sqrt{2}} \langle \varphi_\one,\vec{f}^{\cP}(T_s)\rangle \Big) = (1+o(1))\erf\Big(\frac{2}{\pi} e^{ - s}\Big)\,.
\end{align*}
Again the $o(1)$ term above is uniform over $\cP \in \Gamma$. Combining the three terms above and taking $n\to\infty$, we obtain the desired lower bounding cutoff profile.  
\end{proof}
	
	It remains to prove Lemma~\ref{lem:stationary-with-shift}, which we used in the above proof. 
	\begin{proof}[\textbf{\emph{Proof of Lemma~\ref{lem:stationary-with-shift}}}]
	    Consider the identity coupling (recall Remark~\ref{rem:identity-coupling}) of the DGFF dynamics $\bh^\bg$ to the dynamics initialized from $\pi$. Stationarity of $\pi$ implies that for every~$t \ge 0$, 
\begin{align*}
    \bh^\cP(t) - \vec{f}^\cP(t) \sim \pi\,,
\end{align*}
where $\vec{f}_x^\cP(t)$ is the deterministic discrete heat equation process defined as follows: initialize $\vec{f}^\cP(0)= \vec{l}$, and if the clock at a site $y$ rings in $\cP$ at time $t$, then let $f_y^\cP(t) = \frac{1}{4} \sum_{z\sim y} f_z^\cP(t_-)$. By analogous reasoning to that used in Section~\ref{sec:random-walk-representation}  (indeed this process is the same as a DGFF dynamics in which the Gaussians used in the updates are all set to be zero), one has 
\begin{align*}
    f_x^\cP(t) = \mathbb E [\vec{l}(S_{-t}^{x,\cP})]\qquad \mbox{for all $x\in \Lambda_n$}\,.
\end{align*}
This is exactly the claimed expression for the centering $\vec{f}^\cP(t)$. 
	\end{proof}

	\appendix
	\section{Standard estimates for simple random walks}\label{sec:appendix}
	In this appendix, we include all the estimates we laid out in Section~\ref{subsec:random-walk-prelims} on the discrete-time simple random walk on $\Lambda_n$ absorbed at $\partial \Lambda_n$. 
	
	\subsection{Spectral theory of the random walk on $\Lambda_n$}
	
	Consider the discrete Laplacian  $\Delta_n$ on $\Lambda_n$ as defined in~\eqref{eq:Laplacian}. 
	This is the same as the Dirichlet Laplacian (i.e., domain restricted to functions which are identically zero on $\partial \Lambda_n$) on the graph $\Lambda_n \cup \partial \Lambda_n$. 
	
	We will find that the eigenvectors of $-\Delta_n$ are $2$-way products of the eigenvectors on the line-segment $\llb 1,n-1\rrb$. More precisely, the eigenvectors of $-\Delta_n$ are indexed by pairs $\mathbf{i} = (i_1,i_2)$ and if $x=(x_1,x_2)\in \Lambda_n$, they are given by 
	\begin{align*}
	\varphi_{\mathbf{i}} (x) : = \varphi_{i_1}(x_1)\varphi_{i_2}(x_2) \qquad \mbox{where} \qquad \varphi_{i}(k) = \sqrt{\frac{2}{n}}\sin \Big(\frac{k i \pi}{n}\Big)\,. 
	\end{align*}
	Notice that these functions take the value $0$ at $k =0$ and $k=n$, so that for every $\mathbf i$, $\varphi_i$ is identically zero on $\partial \Lambda_n$, consistent with the Dirichlet boundary conditions. 
	
	Similarly, we will find that the eigenvalues corresponding to these eigenvectors are given by 
	\begin{align*}
	\lambda_{\mathbf i} = \frac{1}{2} (\lambda_{i_1} + \lambda_{i_2}) \qquad \mbox{where} \qquad \lambda_i = 1- \cos\Big(\frac{i \pi}{n}\Big)\,.
	\end{align*}
	
	\begin{fact}\label{fact:eigenfunctions-eigenvalues}
		The functions $\varphi_{\mathbf{i}}$ as defined above form an orthonormal basis of eigenvectors for $-\Delta_n$, and their corresponding eigenvalues are $\lambda_{\mathbf{i}}$. 
	\end{fact}
	
	\begin{proof}
		We will use the fact that $(\varphi_i)_{i=1}^{n-1}$ are the normalized eigenfunctions on the line segment $\llb 1,n-1\rrb$ with corresponding eigenvalues $\lambda_i$: see e.g.,~\cite[Section A.4]{Lacoin-adjacent-transposition},~\cite[Section 8.2]{Lawler-RW-book}. Then the fact that $\varphi_{\mathbf i}$ are eigenfunctions of $\Lambda_n$ with eigenvalues $\lambda_{\mathbf{i}}$ follows from the following calculation.  
		\begin{align*}
		\Delta \varphi_{\mathbf i} (x) 
		& = \frac{1}{2} \sum_{j=1,2} \Big(\prod_{k\ne j} \varphi_{i_k}(x_k)\Big) \bigg[\frac 12 \sum_{\epsilon_j \in \{-e_j, e_j\}} \varphi_{i_j} (x_j + \epsilon_j) - \varphi_{i_j} (x_j)\bigg] \\
		& = \Big(\frac{1}{2} \sum_{j=1,2} \lambda_{i_j}\Big) \varphi_{\mathbf i}(x)\,.
		\end{align*}
		The fact that they are orthonormal follows from 
		\begin{align*}
		\sum_{x} \varphi_{\mathbf i}(x) \varphi_{\mathbf j}(x)  = \prod_{k=1,2} \sum_{x_k} \varphi_{i_k}(x_k)\varphi_{j_k}(y_k) = \prod_{k=1,2} \mathbf 1\{i_k = j_k\} = \mathbf 1\{\mathbf i = \mathbf j\}\,.
		\end{align*}
		Finally, the fact that this gives all the eigenfunctions of the Laplacian follows from a simple counting argument since this gives $(n-1)^2$ many orthogonal vectors, which equals $|\Lambda_n|$; at the same time $\Delta_n$ is evidently a $|\Lambda_n|\times |\Lambda_n|$ matrix. 
	\end{proof}

	\begin{fact}\label{fact:eigenvalue-lower-bound}
		We have for every $\mathbf{i}\in \{1,...,n-1\}^2$, $\lambda_{\mathbf{i}} \ge \Big(\frac{\sum_{j} i_j}{2}\Big) \lambda_{\mathbf{1}}$.
	\end{fact}
	\begin{proof}
		A simple calculation (see also~\cite{Lacoin-adjacent-transposition}) gives $\lambda_i \ge i \lambda_1$ for all $i$. Then 
		\[    \lambda_{\mathbf i} = \frac{1}{2} \sum_{j=1,2} \lambda_{i_j} \ge \frac 12 \sum_{j=1,2} i_j \lambda_1 = \Big(\frac{1}{2} \sum_{j=1,2} i_j\Big)\lambda_{\mathbf 1}\,. \qedhere
		\]
	\end{proof}
	
	\begin{proof}[\textbf{\emph{Proof of Fact~\ref{fact:order-of-magnitudes}}}]
		We begin with the expansion for $\lambda_\one$. By Taylor expanding, $\lambda_\one = \lambda_1  = \frac{\pi^2}{2n^2} + O(n^{-4})$. 
		The upper bounds on $\max_{x\in \Lambda_n}\varphi_\one(x)$ and $\max_{x\in \Lambda_n}\widehat \varphi_\one(x)$ follow by bounding sine by $1$. The lower bound on $\widehat \varphi_\one(x)$ for $x$ such that $d(x,\partial \Lambda_n)\ge n/4$ comes from the fact that $\sin(k\pi/n)$ is positive for all $k$, and $\Omega(1)$ for all $k\in [n/4,3n/4]$.  
	\end{proof}
	
	\subsection{Survival probability for the SRW on $\Lambda_n$}
	Here, we establish that the survival probability for the simple random walk on $\Lambda_n$ decays exponentially with rate $\lambda_\one$, and moreover, the constant in front of that decay is given by the normalized value of the top eigenvector of $-\Delta_n$.

	\begin{proof}[\textbf{\emph{Proof of Lemma~\ref{lem:survival-probability-srw}}}]
		We begin with the discrete-time bound. To express the survival probability  $\mathbb E[\mathbf 1\{Y_k^x \notin \partial \Lambda_n\}]$ let  $f = \mathbf 1_{\{\Lambda_n\}}$. Consider 
		\begin{align*}
		P^k f(x) := \mathbb E[\mathbf 1\{Y_k^x \in \Lambda_n  \}]\,.
		\end{align*}
		where $P$ is the transition matrix for the killed random walk on $\Lambda_n$. 
		Since $P = I + \Delta$, expanding this in the eigenbasis of the Laplacian,
		\begin{align*}
		P^k f(x) = \sum_{\mathbf i} (1-\lambda_{\mathbf i})^k \langle \mathbf 1_{\{\Lambda_n  \}}, \varphi_{\mathbf i} \rangle \varphi_{\mathbf i}(x)\,,
		\end{align*}
		where the inner product is in $\ell^2(\Lambda_n)$, i.e., $\langle f,g\rangle = \sum_{x\in \Lambda_n} f(x)g(x)$. 
		For the upper bound, using the inequality $1-x\le e^{-x}$, and the positivity of $\widehat \varphi_\one$, we have the upper bound 
		\begin{align*}
		P^k f(x) \le  \widehat \varphi_\one(x) (1-\lambda_\one)^{k}  + \sum_{\mathbf i\ne \one} e^{ - \lambda_{\mathbf i} k} |\langle \mathbf 1_{\{\Lambda_n  \}}, \varphi_{\mathbf i} \rangle| |\varphi_{\mathbf i}(x)|\,.
		\end{align*}
		We also have the lower bound 
		\begin{align*}
		P^k f(x) & \ge \widehat \varphi_\one(x) (1-\lambda_\one)^k \  - \sum_{\mathbf i\ne \mathbf 1} e^{ - \lambda_{\mathbf i} k} |\langle \mathbf 1_{\{\Lambda_n  \}}, \varphi_{\mathbf i}\rangle| | \varphi_{\mathbf i}(x)| \\
			& \ge \widehat \varphi_\one (x) (1-\lambda_\one)^k - \sum_{\mathbf i\ne \mathbf 1} e^{ - \lambda_{\mathbf i} k} |\langle \mathbf 1_{\{\Lambda_n  \}}, \varphi_{\mathbf i}\rangle| | \varphi_{\mathbf i}(x)|\,.
		\end{align*}
		
		It remains to show that the contribution from the $\mathbf i \ne \one$ terms are $o(e^{ - \lambda_\one k})$. Towards that, let $\widehat \varphi_{\mathbf i}(x) : = \varphi_{\mathbf i}(x) \sum_{y} \varphi_{\mathbf i}(y)$, and notice that by the same reasoning as for $\widehat \varphi_{\mathbf 1}(x) =O(1)$ from Fact~\ref{fact:order-of-magnitudes},  we get $\widehat \varphi_{\mathbf i}(x) = O(1)$ uniformly over all $x\in\Lambda_n$. Thus, we have that for all $x\in \Lambda_n$,  
		\begin{align*}
		\sum_{\mathbf i\ne \mathbf 1} e^{ - \lambda_{\mathbf i} k} |\langle \mathbf 1_{\{\Lambda_n  \}}, \varphi_{\mathbf i}\rangle| | \varphi_{\mathbf i}(x)| \le C \sum_{\mathbf i\ne \mathbf 1} e^{ - \lambda_{\mathbf i} k}\,.
		\end{align*}
		Now, by Fact~\ref{fact:eigenvalue-lower-bound}, we have 
		\begin{align*}
		\sum_{\mathbf i\ne \one} e^{- (\frac{1}{2} \sum_{j=1,2} i_j )\lambda_{\mathbf 1}k}  = \sum_{r \in \frac{1}{2} \mathbb Z: r\ge 1+ \frac 12} \,\, \sum_{\mathbf i: \frac{1}{2} \sum_{j=1,2} i_j = r}  e^{ - r \lambda_{\mathbf 1} k}  \le   \sum_{r \in \frac{1}{2} \mathbb Z: r\ge 1+ \frac 12} (2r)^2 e^{ - r\lambda_{\mathbf 1}k}\,.
		\end{align*}
		Using, say, summation by parts, we find that the above is at most 
		$$C ( r^2 + {(\lambda_{\mathbf 1} k)^{-2}}) e^{- r\lambda_{\mathbf 1}k}\Big|_{r=1+\frac 12}^{n}\le Ce^{ - \lambda_\one k} e^{ - \lambda_\one k/2}\,.$$ 
		This is $o(e^{ - \lambda_\one k})$ as desired, since $\lambda_{\mathbf 1}k$ is diverging when $n^2 = o(k)$ by Fact~\ref{fact:order-of-magnitudes}.
		
		The continuous-time bound follows essentially the same reasoning, with $P^k$ replaced by $e^{ - t \Delta}$, which automatically replaces $(1-\lambda_{\mathbf{i}})^k$ by $e^{ - \lambda_{\mathbf{i}}k}$ in the above.
	\end{proof}

	\bibliographystyle{abbrv}
	\bibliography{references}

\end{document}